%% file: sarksubrev.tex
\documentclass[11pt]{amsart}

\usepackage{amsmath,amssymb,latexsym,amsthm,newlfont,enumerate}

\usepackage{pst-plot}
\usepackage{pstricks}
\usepackage{pstricks-add}

\usepackage{graphicx}
\usepackage[all]{xy}
\makeindex
\input{macros}
\begin{document}
\title[Relations in the Sarkisov Program]{Relations in the Sarkisov Program}

\author{Anne-Sophie Kaloghiros}
\address{Department of Mathematics, Imperial College London, 180 
Queen's Gate, London SW7 2AZ, UK}
\email{a.kaloghiros@imperial.ac.uk}
\thanks{The author is supported by the Engineering and Physical Sciences Research Council [grant number EP/H028811/1]. Part of this research was carried out while I was visiting the University of Illinois at Chicago. I wish to thank A. Corti for suggesting this problem, P. Cascini, S. Lamy and V. Lazi\'c for valuable suggestions. While working on this article, I attended a talk by Professor Shokurov in which similar structures for Sarkisov relations were announced. It is my pleasure to acknowledge Professor Shokurov's influence on this work, however my results are slightly different and not as strong as those he announced.}
\begin{abstract}The Sarkisov Program studies birational maps between varieties that are end products of the Minimal Model Program (MMP) on nonsingular uniruled varieties. If $X$ and $Y$ are terminal $\Q$-factorial projective varieties endowed with a structure of Mori fibre space, a birational map 
$f\colon X\dashto Y$ is the composition of a finite number of \emph{elementary Sarkisov links}. This decomposition is in general not unique: two such define a \emph{relation in the Sarkisov Program}. 
I define \emph{elementary relations}, and show they generate relations in the Sarkisov Program. Roughly speaking, elementary relations are the relations among the end products of suitable relative MMPs of $Z$ over $W$ with $\rho(Z/W)=3$.  
\end{abstract}
\maketitle
 \section{Introduction}\label{intro}
 Let $Z$ be a nonsingular projective variety. Conjecturally, the \emph{Minimal Model Program (MMP) terminates} for $Z$, and there is a finite sequence of elementary birational operations \emph{directed by $K_Z$} 
\[\vphi\colon Z=X_0\dashto X_1 \dashto \dots \dashto X_n=X,\]
where $X$ is terminal and $\Q$-factorial, and one of:
\begin{enumerate}
\item[1.] \emph{$X$ is a minimal model}: $K_X$ is nef, 
\item[2.] \emph{$X$ is a Mori fibre space}: $X$ is endowed with a fibration morphism $p\colon X\to S$ such that $\rho(X/S)=1$, and ${-}K_X$ is $p$-ample.
\end{enumerate}
The map $\vphi$ is \emph{the result of a $K_Z$-MMP}, and $X$ is a \emph{distinguished representative} in the birational equivalence class of $Z$. However, neither $\vphi$ nor $X$ are unique. Since the birational classification of varieties is one of the primary goals of the MMP, it is natural to study birational maps between distinguished representatives that are results of the $K_Z$-MMP. The type of $X$ (minimal model or Mori fibre space) depends on whether $Z$ is uniruled or not, so that one only needs to consider birational maps between minimal models or between Mori fibre spaces. In fact, if $X/S$ and $Y/T$ are birational Mori fibre spaces, there is a smooth projective variety $Z$ equipped with morphisms $Z\lto X$ and $Z\lto Y$ that are the results of $K_Z$-MMPs. 
The \emph{Sarkisov Program} studies birational Mori fibre spaces, or, equivalently, birational maps between distinguished models produced by the MMP on nonsingular uniruled varieties.

More generally, one can ask whether birational maps between distinguished representatives produced by the MMP on a nonsingular variety $Z$ (or by the log-MMP on a klt pair $(Z, \Delta)$) admit any particular structure. Since any result of a $K_Z$-MMP is the composition of a finite number of elementary maps (contractions of extremal rays), any such birational map $X\dashto Y$ is the composition of a finite number of elementary steps $\vphi_k\colon X_k\dashto X_{k+1}$ for $k=0, \dots , N$, where $X_0\simeq X$, $X_N\simeq Y$ and for each $k$, $\vphi_k$ or $\vphi_k^{-1}$ is the contraction of an extremal ray.   A drawback of this tautological decomposition is that the intermediate varieties $X_k$ are not in general distinguished representatives.  The following two theorems give a decomposition where each intermediate step is a birational map between minimal models or between Mori fibre spaces. 
\begin{thm}\cite{Kaw08}\label{thm:0.1}
A birational map $f\colon X\dashto Y$ between minimal models is the composition of a finite number of flops. 
\end{thm}
\begin{thm}\cite{HM09}\label{thm:0.2}
A birational map  $f\colon X/S\dashto Y/T$ between Mori fibre spaces is the composition of a finite number of elementary Sarkisov links. 
\end{thm}
An \emph{elementary Sarkisov link} $L_{i,j}\colon X_i/S_i\dashto X_j/S_j$ is the birational map between two Mori fibre spaces that are end products of the MMP of $Z$ over $W$, where $Z$ is terminal, $Z\lto W$ is a fibration ($\dim W<\dim Z$) with $\rho(Z/W)=2$ (see Section~\ref{prelim}). 
Theorem~\ref{thm:0.2} is proved for $3$-folds in \cite{Co95}, where a 
decomposition of the birational map $f\colon X/S\dashto Y/T$ is achieved by \emph{untwisting} $f$. Explicitly, there is a mobile linear system $\mcal L$ on $X$ associated to $f$, and an elementary link $L_{1,2}\colon X= X_1\dashto X_2$ that partially resolves the highest multiplicity locus of $\Bs\mcal L$; replace $X$ by $X_2$, $f$ by $f\circ L_{1,2}^{-1}$ and $\mcal L$ by its image on $X_2$ and repeat the procedure. After a finite number of iterations, the image of $\mcal L$ becomes basepoint free and $f$ has been decomposed in elementary links. The proof in arbitrary dimension uses different ideas. Both Theorems~\ref{thm:0.1} and ~\ref{thm:0.2} are consequences of the study of the \emph{geography of ample models} of \emph{effective} adjoint divisors $K_Z+\Theta$, as $\Theta$ varies in a suitable region of the space of $\R$-divisors $\Div_\R(Z)$.

This paper studies \emph{relations in the Sarkisov Program}.  
Given a birational map $f\colon X\dashto Y$ between Mori fibre spaces, the decomposition of $f$ into elementary Sarkisov links in Theorem~\ref{thm:0.2} is not unique. If
\[ f\colon X_1/S_1\dashto \ldots \dashto X_r/S_r \mbox{ and } f\colon X_1/S_1\dashto \ldots \dashto X_q/S_q\] are two distinct decompositions, then 
\begin{equation}\label{eq:1.1}X_1/S_1\dashto \ldots \dashto X_r/S_r\simeq X_q/S_q\dashto X_r/S_r \dashto  X_1/S_1\end{equation}
is an element of $\Aut(X_1/S_1)$, i.e.~an automorphism of $X=X_1$ that commutes with the fibration $X_1\to S_1$.  
A \emph{relation in the Sarkisov Program} is a decomposition of an element of $\Aut(X_1/S_1)$ into elementary Sarkisov links as in \eqref{eq:1.1}. 
Relations arise naturally among Sarkisov links between Mori fibre spaces produced by the MMP on suitable varieties $Z$ with $\rho(Z)\geq 3$. 

As an example, let $S$ be the del Pezzo surface obtained by blowing up two points $P_1$ and $P_2$ on $\PS^2$. Let $E_1,E_2$ be the $(-1)$-curves on $S$ and $L$ the proper transform of the line through $P_1$ and $P_2$. Then, $\Effb(S)$, the pseudo effective cone of $S$, is the cone over a polytope $\mcal P$ as in Figure~\ref{fig:1.1}.
\begin{figure}[h]\label{fig:1.1}\vspace{-20pt}
\begin{center}
\begin{pspicture}[psscale=.5](-0.5,0)(3.6, 3)\psset{unit=0.75cm}
\psdots(0,0)(1,1.5)(2,3)(3,1.5)(4,0)(2,1)
\rput[b](2, 3.2){$L$}
\rput[r](-0.2, 0){$E_1$}\rput[l](4.2, 0){$E_2$}
\rput[br](0.8,1.5){$L+E_1$}
\rput[bl](3.2,1.5){$L+E_2$}
\rput(2,0.5){$\mcal C_4$}
\rput(2,1.3){$\mcal C_0$}
\rput(2,1.9){$\mcal C_1$}
\rput(1.2,1){$\mcal C_2$}
\rput(2.8,1){$\mcal C_3$}
\psline(0,0)(2,3)(4,0)(0,0)
\psline[linewidth=0.4pt](0,0)(3,1.5)(1,1.5)(4,0)
\end{pspicture}
\vspace{-10pt}
\end{center}\caption{}\vspace{-10pt}
\end{figure}
There is a fan of $\Effb S= \cup \mcal C_i$, where $\mcal C_i$ are rational polyhedral cones, that determines a \emph{geography of ample models}. Mori fibre spaces produced by $K_S$-MMPs are in $1$-to-$1$ correspondence with those facets of the $\mcal C_i$ lying in the locus of non big divisors. In Figure 1, these are the facets of the $\mcal C_i$ supported on the sides of the triangle, e.g~$[L+E_1, L]$ represents $\PS_a^1\times \PS_b^1\to \PS_a^1$, $[L+E_2, L]$  represents $\PS_a^1\times \PS_b^1\to \PS_b^1$ and $[E_1, l+E_1]$ represents $\F_{1a}\to \PS^1_a$ (see Example~\ref{exa:4.1} for details). Hacon and McKernan show that if two of these line segments have non-empty intersection (here at the points $L,L+E_1, L+E_2, E_1$ and $E_2$), there is an elementary Sarkisov link between the corresponding Mori fibre spaces. Further, a decomposition into elementary links of any birational map $f\colon X/S\dashto Y/T$ can be obtained by tracing a path on $\partial \mcal P$ from a point in the interior of the line segment corresponding to $X/S$ to a point in the interior of the line segment corresponding to $Y/T$. In the same way, a relation in the Sarkisov Program \emph{dominated by S} is represented by a loop on $\partial \mcal P$, i.e.~by an element of the fundamental group $\pi_1(\mcal P)= \pi_1(\partial \Effb(S)\smallsetminus\{0\})$. Here,  $\pi_1(\mcal P)= \Z \cdot \gamma$, where $\gamma$ represents the relation:
\begin{equation}
\label{eq:1.2}
\xymatrixrowsep{0.15in}
\xymatrixcolsep{0.2in}
 \xymatrix{
S\ar@{=}[r]\ar[d]&S\ar[d]&&S\ar[d]\ar@{=}[r]& S\ar[d]&\\
\PS^1_a\times\PS^1_b \ar[d] & \F_{1b}\ar[d] \ar@{=}[r] & \F_{1b}, \F_{1a}\ar@{=}[r]\ar[d] & \F_{1a}\ar[d] & \PS^1_a\times\PS^1_b\ar[d] \ar@{=}[r]&\PS^1_a\times\PS^1_b\ar[d]  \\
 \PS^1_b \ar@{=}[r]& \PS^1_b\ar[d]& \PS^2\ar[d] & \PS^1_a\ar[d]\ar@{=}[r]& \PS^1_a\ar[d] &\PS^1_b\ar[d]\\
& \{P\}\ar@{=}[r]&\{P\}\ar@{=}[r]&\{P\}&\{P\}\ar@{=}[r]&\{P\}&} \end{equation}

I call a relation as in \eqref{eq:1.1} \emph{elementary} if there is a variety $W$ and morphisms $X_i\to S_i\lto W$ that commute with the links $L_{i, i+1}$ with $\rho(X_i/W)\leq 3$ for all $i$ (see Section~\ref{relations} for a precise definition). If $Z$ is a common resolution of all Sarkisov links $L_{i,i+1}$, an elementary relation corresponds to a simplicial loop on the boundary of the pseudoeffective cone of $Z$ and can be thought of as a \emph{relative Picard rank $3$ MMP}. It is naturally dual to 
the boundary of a suitable $2$-dimensional polytope as in figure~\ref{fig:1.1}. 

I show that relations in the Sarkisov Program correspond to simplicial loops on suitable polyhedral complexes supported on the locus of strictly pseudoeffective (effective non-big) divisors of nonsingular uniruled varieties. The following theorem (Theorem~\ref{thm:4.1}) is then a straightforward application of Van-Kampen's theorem. 
\begin{thm}\label{thm0}
A relation in the Sarkisov Program is the composition of a finite number of elementary relations. 
\end{thm}
There are $4$ \emph{types} of elementary Sarkisov links that correspond to possible configurations in the \emph{$2$-ray game}, or, loosely speaking, to configurations of extremal rays on a (weak) Fano variety $Z$ of Picard rank $2$. There is no such result when $\rho(Z)=3$, and  an explicit \emph{classification} of elementary relations is therefore not a reasonable goal. However, one can define two types (A and B) of elementary relations consisting of an arbitrary number of elementary links (see Definition~\ref{dfn:4.3}). For example, \eqref{eq:1.2} is of type A. In Example~\ref{exa:4.2}, I determine all elementary relations that arise among the Mori fibre spaces that are end products of the MMP on Fano $3$-folds $Z$ with $\rho(Z)=3$, and show that both types of relations occur. In fact, both types of relations are realized among the end products of the MMP on Picard rank $3$ toric Fano $3$-folds, or in the Cremona group of $\PS^3$.   

\subsection*{Outline}
I sketch the main idea underlying the results in \cite{HM09} and in this paper. For now, I do not distinguish between divisors and their numerical equivalence classes. 
Let $f\colon X\dashto Y$ be a birational map, and assume that $\vphi_X\colon Z\lto X$ and $\vphi_Y\colon Z\lto Y$ are results of the MMP on a nonsingular uniruled variety $Z$. Denote by $p_X\colon X\to S$ and $p_Y\colon Y\to T$ the Mori fibrations on $X$ and $Y$. If $D$ is an effective $\Q$-divisor on $Z$, $Z\dashto Z_D\simeq \Proj R(Z, D)$ (when it makes sense) is the \emph{ample model of $D$}--the precise definition is recalled in Section~\ref{prelim}.  
Define klt pairs $(Z,\Theta_X), (Z, \Theta_Y)$, where $\Theta_X, \Theta_Y$ are ample divisor such that $\vphi_X$ (resp.~$\vphi_Y$) is the ample model of $K_Z+(1+\varepsilon)\Theta_X$ (resp. $K_Z+(1+\varepsilon)\Theta_Y $) for $0<\varepsilon <\!<\!<1$ and $p_X\circ \vphi_X$ (resp. $p_Y\circ \vphi_Y$) that of $K_Z+\Theta_X$ (resp.~$K_Z+\Theta_Y$). Then, for suitable $\Theta_X$ and $\Theta_Y$, there is a rational polyhedral cone $\mcal{C} \subset \Effb_\R(Z)$, such that the following holds. Every divisor $D\in \mcal C$ is effective, $K_Z+\Theta_X$ and $K_Z+\Theta_Y$ lie on $\partial^+ \mcal C= \mcal C\cap (\Effb_\R(Z)\smallsetminus \B_\R(Z))$, while $K_Z+(1+\varepsilon)\Theta_X$ and $K_Z+(1+\varepsilon)\Theta_Y$ lie in the interior of $\mcal C$. 
Section~\ref{prelim} recalls general results on cones such as $\mcal C$. In particular, there is a decomposition of $\mcal C= \coprod \mcal A_i$ into relatively open rational polyhedral cones that determines a \emph{geography of ample models}. Namely, all $\Q$-divisors $D\in \mcal A_i$ have the same ample model and birational maps between these ample models are determined by incidence relations between the closed cones $\overline{\mcal A_i}$. The Sarkisov program studies incidence relations between cones $\overline{\mcal A_i}$ that intersect the locus of non-big divisors $\partial^+ \mcal C$. 

Given a relation as in \eqref{eq:1.1},  one can define a suitable cone $\mcal C$ in $\Div_{\R}(Z)$, where $Z$ is a common resolution of all Sarkisov links in \eqref{eq:1.1}, 
and a polyhedral complex $\mcal B^+(\mcal C_\mathfrak R)$ associated to the geography of ample models on $\mcal C$--this is done in Section~\ref{constructions}. 
There is a simplicial complex $\mcal N_\mathfrak R$-- the \emph{nerve of $\mcal B^+(\mcal C_\mathfrak R)$}-- that encodes the \emph{Sarkisov Program dominated by $Z$}. Explicitly, there is a $1$-to-$1$ correspondence between vertices of $\mcal N_\mathfrak R$ and Mori fibre spaces $X_i/S_i$, and between edges of $\mcal N_\mathfrak R$ and elementary Sarkisov links between the $X_i/S_i$ \emph{dominated} by $Z$. Given a face $\mcal A\subseteq \partial^+ \mcal C$ in the decomposition, one can construct a \emph{residual cone} $\resi_{\mcal A}\mcal C$ of dimension $d=\dim \mcal C- \dim \mcal A$. There is an induced geography of ample models on $\resi_\mcal A \mcal C$, and an induced complex $\resi_\mcal A \mcal N_\mathfrak R$.  If $\vphi\colon Z\lto W$ is the ample model of divisors $D\in \mcal A$, the geography encodes the birational geometry of $Z$ over $W$, and $\resi_\mcal A \mcal N_\mathfrak R$ the Sarkisov program among results of the MMP over $W$.  

In Section~\ref{relations}, I show that
relations in the Sarkisov program correspond to edge-loops on $\mcal N_\mathfrak R$, i.e.~elements of $\pi_1(\mcal N_\mathfrak R)$, or, equivalently, of $\pi_1(\partial^+ \mcal C\setminus \{0\})$. Theorem~\ref{thm0} is proved by constructing a cover of $\mcal N_\mathfrak R$ by residual complexes $\resi_\mcal A \mcal N_\mathfrak R$.

 \section{Preliminary Results}\label{prelim}
I consider varieties defined over $\C$, and denote by $\R_+$ and $\Q_+$ the set of non-negative real and rational numbers.
\subsection{Convex Geometry and polyhedral complexes}\label{cogeco}
 The topological closure of $\mcal S\subset \R^N$ is denoted by $\overline{\mcal S}$, and the boundary of a closed $\mcal C\subset \R^N$ by $\partial \mcal C$.  

A \emph{convex polyhedron} in $\R^N$ is defined by a finite number of (strict or not) linear inequalities in $\R^N$. A \emph{rational polytope} in $\R^N$ is a compact set which is the convex hull of finitely many rational points in $\R^N$. A \emph{rational polyhedral cone} in $\mathbb R^N$ is a convex cone spanned by finitely many rational vectors. The dimension of a cone in $\R^N$ is the dimension of the minimal $\R$-vector space containing it.
   
I recall some standard properties of convex sets which will be used repeatedly (see \cite{Grun} for references). 
Let $K\subset \R^N$ be a closed convex set. There is a unique subspace $L\subset \R^N$ of maximal dimension such that $K$ contains a translate of $L$. If $L^*\subset \R^N$ is a subspace complementary to $L$, then $K= L+(L^*\cap K)$ and $L^*\cap K$ contains no line.

If $K\subset \R^N$ is a closed convex subset that contains no line, there is a hyperplane $H$ such that $H\cap K$ is compact and has dimension $\dim K-1$.
In particular, if $K\subset \R^N$ is a closed cone with apex $x_0$, $K= \Cone_{x_0}(H\cap K)$.

\begin{dfn}\label{dfn:2.1}
A \emph{polyhedral complex} $\mcal C$ in $\R^N$ is a small category whose objects are convex polyhedra in $\R^N$, whose morphisms are isometric morphisms and that satisfies:
\begin{enumerate}
\item[(i)] If $c_1\in \Ob \mcal C$ and $c_2$ is a face of $c_1$, then $c_2\in \Ob \mcal C$, and the inclusion $i \colon c_2\lto c_1$ is a morphism of $\mcal C$.
\item[(ii)] If $c_1,c_2\in \Ob \mcal C$, there is at most one morphism $f\in \Mor \mcal C$ such that $f(c_1)\subset c_2$.
\end{enumerate}
The \emph{faces} of $\mcal C$ are the elements of $\Ob \mcal C$. A face $c$ of  $\mcal C$ is a \emph{facet} if $f(c)= c'$ for all $\{f\colon c\lto c'\}\in \Mor \mcal C$.
The complex $\mcal C$ is \emph{pure of dimension $n$} if all its facets have dimension $n$.
The \emph{underlying space} of $\mcal C$ is the topological space $|\mcal C|=\coprod_{c\in \Ob \mcal C} c/\sim$ where $c\sim f(c)$ for all $f\in \Mor \mcal C$ and $c\in \Ob \mcal C$. 
\end{dfn}
\begin{dfn}
Let $\mcal C$ be a pure polyhedral complex of dimension $n$. The \emph{Nerve of $\mcal C$} is the simplicial complex $\Ner(\mcal C)$ defined by:
\begin{enumerate} 
\item[(i)]
vertices of $\Ner \mcal C$ are dual to facets of $\mcal C$; if $c$ is a facet of $\mcal C$, $c^*$ denotes the corresponding vertex of $\Ner \mcal C$.   
\item[(ii)]vertices $c_0^*, \dots , c_k^*$ span a $k$-simplex if there is a face $c$ of $\mcal C$ of dimension $n-k$ and a collection of morphisms $\{f_i \colon c\lto c_i\}\in \Mor \mcal C$. The $k$-simplex $\sigma= [c^*_0, \dots, c^*_k]$ is \emph{dual} to $c$. 
\end{enumerate}
\end{dfn}
\begin{rem}\label{rem:2.1} For $k<n$, a $k$-dimensional face $c$ of $\mcal C$ needs not have a dual $k$-simplex; when it does, there may be several dual simplices. When $n=1$, $\mcal C$ is simplicial and dual to $\Ner \mcal C$.
 \end{rem}
\begin{dfn} \label{dfn:4.1}Let $K$ be a simplicial complex and $K^{(m)}$ its $m$-skeleton. 

An \emph{edge path in K} is a finite sequence $(v_0, v_1, \dots, v_n)$ of vertices of K such that $\{v_i, v_{i+1}\}$ spans a simplex for each $i$. An \emph{edge loop} is 
an edge path with $v_0=v_n$; the product of $(v_0, \dots, v_n)$ and $(v_n,\dots, v_m)$ is $(v_0, \dots, v_m)$. 

Edge paths are \emph{equivalent} if they are related by a finite sequence of \emph{elementary equivalences} given by $(\dots, v_i, v_{i+1}, v_{i+2}, \dots)\sim (\dots, v_i, v_{i+2}, \dots)$ when $\{ v_i, v_{i+1}, v_{i+2}\}$ spans a simplex of $K$. 

The \emph{edge path group} $E(K,v)$ is the group of equivalence classes of edge loops based at $v$; if $|K|$ is a geometric realisation of $K$, $E(K, v)\simeq \pi_1(|K|, v)$. 
\end{dfn}

\begin{lem}\label{lem:2.0} If $\mcal C$ is a finite pure polyhedral complex of dimension $n$, $|\mcal C|$ and $|\Ner \mcal C|$ are homotopy equivalent.
\end{lem}
\begin{proof}
This is \cite[Lemma 10]{KK11}. 
\end{proof}
\subsection{Models of effective divisors and the classical MMP}\label{models}
Let $Z$ be a normal projective variety and $\mathbf R\in \{ \Z, \Q, \R\}$.  The group of $\mathbf R$-Cartier $\mathbf R$-divisors is denoted by $\Div_\mathbf R(Z)$, and $\sim_{\mathbf R}$ and $\equiv$ denote $\mathbf R$-linear and numerical equivalence of $\mathbf R$-divisors. If $Z\lto X$ is a morphism to a normal projective variety, numerical equivalence over $X$ is denoted by $\equiv_X$. Let $\Pic(Z)_\mathbf R=\Div_\mathbf R(Z)/\sim_\mathbf R$ and $N^1(Z)_\mathbf R=\Div_\mathbf R(Z)/\equiv$. Denote by $\rho(Z)= \dim N^1(Z)_{\Q}$ the \emph{Picard rank} of $Z$.
If $D\in \Div_\mathbf R(Z)$, $[D]\in N^1(Z)_\mathbf R$ is the image of $D$ under the natural map $\Div_\mathbf R(Z)\to N^1(Z)_\mathbf R$.   

The ample, big, nef, effective, and pseudo-effective cones in $N^1(Z)_\R$ are denoted by $\Amp (X)$, $\B(Z)$, $\Nef (Z)$, $\Eff(Z)$, and $\overline{\Eff}(Z)$. 

A {\em pair} $(Z,\Delta)$ consists of a normal projective variety $Z$ and an effective $\R$-divisor $\Delta$ on $Z$ such that $K_Z+\Delta$ is $\R$-Cartier; $K_Z+\Delta$ is an \emph{adjoint divisor}. I say that $(Z, \Delta)$ is $\Q$-factorial if $Z$ is; I often assume that $K_Z+\Delta$ is $\Q$-Cartier. The pair $(Z,\Delta)$ has \emph{klt} (resp.~\emph{log canonical}) singularities if for every log resolution $f\colon W\longrightarrow Z$, the coefficients of the divisor $K_W-f^*(K_Z+\Delta)$ are all $>{-}1$ (resp.~$\geq{-}1$). A klt pair $(Z, \Delta)$ is \emph{terminal} if, in addition, all $f$-exceptional divisors appear in $K_W-f^*(K_Z+\Delta)$ with $>0$ coefficients.

I recall the definitions of models of divisors introduced in \cite{BCHM}, and show that the results of the classical (log)-MMP are models of suitable effective divisors.
\begin{dfn} \label{dfn:2.3}
Let $Z$ be a normal projective variety and $D\in \W_{\Q}(Z)$. Let $f \colon Z \dashto X$ be a birational contraction such that $D'= f_{*}D$ is $\Q$-Cartier.
\begin{enumerate}
\item[1.] The map $f$ is \emph{$D$-nonpositive} if for a resolution $(p,q)\colon W \rightarrow Z\times X$, 
$$p^*D = q^*D'+ E,$$
where $E\geq 0$ is $q$-exceptional. When $\Supp E$ contains the strict transform of all $f$-exceptional divisors, $f$ is \emph{$D$-negative}.  
\item[2.] When $D$ is effective, $f$ is a \emph{semiample model} of $D$ if $f$ is $D$-nonpositive, $X$ is normal and projective and $D'$ is semiample. 
If $\vphi \colon X\to S$ is the semiample fibration defined by $D'$, the \emph{ample model} of $D$ is $\vphi\circ f\colon Z \dashto X\to S$.
\end{enumerate}
\end{dfn}

\begin{lem}\label{lem:2.1}
Let $Z$ be a $\Q$-factorial projective variety and $D_1$ and $D_2$ be numerically equivalent big $\Q$-divisors. Assume that $R(Z, D_1)$ and $R(Z, D_2)$ are finitely generated and denote by $\vphi_i$ the map $Z \dashto \Proj R(Z, D_i)$.

There is an isomorphism $\eta \colon \Proj R(Z, D_1)\to \Proj R(Z, D_2)$ such that $\vphi_2= \eta\circ \vphi_1$ and $\sB(D_1)= \sB(D_2)$. 
\end{lem}
\begin{proof}This is \cite[Lemma 3.11]{KKL12}.
\end{proof}

\begin{dfn} \label{dfn:2.4}Let $(Z, \Delta)$ be a $\Q$-factorial  klt pair with $K_Z+\Delta\in \Div_\Q(Z)$. 
A birational contraction $\vphi\colon Z \dashto X$ is \emph{the result of a $(K_Z+\Delta)$-MMP} if $X$ is projective, $\Q$-factorial and if:
\begin{enumerate}\item[(i)] $\vphi$ is a semiample model of $K_Z+\Delta$ when $K_Z+\Delta$ is pseudoeffective,
\item[(ii)] $\vphi$ is a semiample model of $K_Z+\Theta$ for some $(Z, \Theta)$ klt with 
$\Theta -\Delta$ nef and $[K_Z+\Theta]\in \partial \Effb_{\Q}(Z)$ otherwise; 
$\vphi\colon Z\dashto X$ is \emph{an MMP with scaling by $\R_+(\Theta-\Delta)$}. 
\end{enumerate}
The MMP terminates for $(Z, \Delta)$ if a contraction $\vphi\colon Z\dashto X$ that is the result of a $(K_Z+\Delta)$-MMP exists. The \emph{classical MMP} refers to the $K_Z$-MMP, where $Z$ is terminal.
\end{dfn}
\begin{rem}\label{rem:2.2}
\cite{BCHM} shows that the MMP terminates for $(Z, \Delta)$ klt unless both $\Delta$ and $K_Z+\Delta$ are strictly pseudoeffective (i.e.~not big).   
\end{rem}

The goal of the following Lemma is to verify that Definition~\ref{dfn:2.4}(ii) is consistent with the classical formulation of the MMP.
\begin{lem}\label{lem:2.2} Let $(Z, \Delta)$ be a $\Q$-factorial klt pair and assume that $K_Z+\Delta$ is not pseudoeffective.
Then, any result $\vphi \colon Z\dashto X$ of a $(K_Z+\Delta)$-MMP is $(K_Z+\Delta)$-nonpositive. 

Let $\Theta$ be a $\Q$-divisor with $\Theta -\Delta$ nef, $[K_Z+\Theta]\in \partial \Effb_{\Q}(Z)$, and such that $\vphi$ is a semiample model of $K_Z+\Theta$. If $f\colon X\to S$ is the Iitaka fibration associated to $\vphi_*(K_Z+\Theta)$, then ${-}\vphi_*(K_Z+\Delta)$ is $f$-nef. 
\end{lem}
\begin{rem} If $Z$ has terminal singularities, and if $\vphi\colon Z\dashto X$ is the result of a $K_Z$-MMP, then $X$ is terminal because $\vphi$ is $K_Z$-nonpositive.\end{rem}
  
\begin{proof}
Let $\vphi\colon Z\dashto X$ be the result of a $(K_Z+\Delta)$-MMP that is a semiample model for $K_Z+\Theta$, where $
\Theta$ is as in Definition~\ref{dfn:2.4}. Let $(p,q)\colon W\lto Z\times X$ be a resolution of $\vphi$. Then, $\vphi_*\Theta$ and $\vphi_*\Delta$ are $\R$-Cartier because $X$ is $\Q$-factorial, and by definition of $\vphi$:
\[ p^*(K_Z+\Theta)= q^*(K_X+ \vphi_*\Theta)+ E,\]
where $E\geq 0$ is $q$-exceptional. 
Since $\vphi$ is a contraction, $p$-exceptional divisors are $q$-exceptional, and there are $q$-exceptional divisors $E_p, E_q\geq 0$ such that:
\[ p_*^{-1}(\Theta- \Delta)= q_*^{-1}(\vphi_*(\Theta-\Delta))= p^*(\Theta-\Delta)- E_p= q^*(\vphi_*(\Theta-\Delta))-E_q.\]
The Negativity Lemma \cite[Lemma 2.19]{Kol92} shows that $E_q\geq E_p$ because $E_p-E_q\equiv_Xp^*(\Theta-\Delta)$. Denote by $E'= E+E_q-E_p\geq E$; then:
\[ p^*(K_Z+\Delta)= q^*(K_X+\vphi_*\Delta)+ E+ E_q-E_p.\]
and $\vphi$ is also $(K_Z+\Delta)$-nonpositive.

Let $C\subset X$ be an irreducible effective curve contracted by $f$, i.e.~with $(K_X+\vphi_*\Theta)\cdot C=0$. Such curves cover $X$ because $K_X+\vphi_*\Theta$ is not big, and we may assume that $C\not \subseteq q(\Exc q)$. 
Then, $(K_X+\vphi_* \Delta)\cdot C= \vphi_*(\Delta- \Theta)\cdot C$. 
Let $\widetilde{C}\subset W$ be an effective curve that maps $1$-to-$1$ to $C$. By the projection formula, $\vphi_*(\Delta- \Theta)\cdot C=q^*(\vphi_*(\Delta-\Theta))\cdot \widetilde{C}$. 
Since $q^*(\vphi_*(\Delta-\Theta))={-}(p^*(\Theta- \Delta)+(E_q-E_p))$ and $(E_q-E_p)\cdot \widetilde{C}\geq 0$ by construction of $C$,  we have $(K_X+ \vphi_* \Delta)\cdot \widetilde C\leq 0$. This shows that  ${-}\vphi_*(K_Z+\Delta)$ is $f$-nef.
\end{proof} 

The following structures arise in the MMP of uniruled varieties. 
\begin{dfn}\label{dfn:2.5}A \emph{log-Mori fibre space} is a projective $\Q$-factorial klt pair $(X, \Delta)$ equipped with a morphism $f\colon X \to S$ such that $\dim S<\dim X$, $\rho(X)= \rho(S)+1$ and ${-}(K_X+\Delta)$ is $f$-ample.

When $\Delta= 0$ and $X$ is terminal, $X/S$ is a \emph{Mori fibre space (Mfs)}. 
\end{dfn} 
By \cite{BCHM}, if $(Z, \Delta)$ is a klt pair with $[K_Z+\Delta]\not \in \Effb(Z)$, there is a contraction $\vphi\colon Z\dashto X$ that is the result of a $(K_Z+\Delta)$-MMP such that the fibration $f\colon X\to S$ has $\rho(X/S)=1$.  Then,  $(X, \vphi_*\Delta)$ is a log-Mori fibre space, because by Lemma~\ref{lem:2.2}, $-(K_X+\vphi_*\Delta)$ is $f$-nef, and $-(K_X+\vphi_*\Delta)\not \equiv_f 0$, so that $-(K_X+\vphi_*\Delta)$ $f$-ample. 

The \emph{Sarkisov Program} studies birational maps between (log-)Mori fibre spaces $(X, \Delta_X)/S$ and $(Y, \Delta_Y)/T$. There is a nonsingular uniruled variety $Z$ (resp.~a klt pair $(Z, \Delta)$ with $K_Z+\Delta$ non-pseudoeffective) such that
$Z\dashto X/S$ and $Z\dashto Y/T$ are results of $K_Z$-MMPs (resp.~$(K_Z+\Delta)$-MMPs).
The simplest examples of birational maps between Mori fibre spaces arise in the \emph{$2$-ray game}. 
\begin{exa}\label{exa0}
Let $X/S$ be a Mori fibre space and $Z\to X$ a divisorial contraction such that ${-}K_Z$ is nef and big over $S$. 
By the Cone theorem, there is another Mori fibre space $Y/T$ that is the result of a $K_Z$-MMP over $S$. Explicitly, $Z\dashto Y$ is either an isomorphism in codimension $1$ followed by a divisorial contraction and $S\simeq T$, or $Z\dashto Y$ is an isomorphism in codimension $1$ and $T\to S$ is a morphism with $\rho(T)= \rho(S)+1$. \end{exa}
\begin{dfn}\label{dfn2.6}\cite{Co95}
An \emph{elementary Sarkisov link} between log-Mori fibre spaces $(X, \Delta_X)/S$ and $(Y, \Delta_Y)/T$ is a diagram 
\begin{equation}\label{esl}
\xymatrixrowsep{0.2in}
\xymatrixcolsep{0.3in}
\xymatrix{\widetilde{X} \ar@{-->}[rr] \ar[d]_f & &\widetilde{Y}\ar[d]^g\\
X\ar[d]_{\vphi} & &Y \ar[d]^{\psi}\\
S\ar[dr]_{p} & & T\ar[dl]^q \\
 & R & } 
\end{equation}
where $\Phi \colon \widetilde{X}\dashto \widetilde{Y}$ is an isomorphism in codimension $1$, $\Phi_* \Delta_X= \Delta_Y$, $\rho(\widetilde{X}/R)= \rho(\widetilde{Y}/R)=2$, and $f,g, p$ and $q$ are either isomorphisms or $(K+ \Delta)$-nonpositive morphisms of relative Picard rank $1$. One of $f$ and $p$ and one of $g$ and $q$ are isomorphisms.
The link is of type I (resp.~III) when $p$ and $g$ (resp.~$f$ and $q$) are isomorphisms and of type II (resp.~IV) when $p$ and $q$ (resp.~$f$ and $g$) are isomorphisms.
\end{dfn}
A Sarkisov link is thus always induced by a $2$-ray configuration as in Example~\ref{exa0}: it is the birational map between two Mori fibre spaces that are end products of the $K_{\widetilde{X}}$-MMP over $W=S,T$ or $R$ with $\rho({\widetilde{X}}/W)=2$ and $\dim W<\dim \widetilde{X}$.

\subsection{Geographies of models on good regions}\label{geog}
 The Sarkisov Program is thus equivalent to the study of birational maps between models of strictly pseudoeffective adjoint divisors on nonsingular uniruled varieties. In this subsection, I recall the construction of some good regions of $\Div_\R(Z)$, for $Z$ nonsingular, where every effective divisor $D$ admits a $\Q$-factorial semiample model $\vphi_D\colon Z\dashto X_D$, and $\vphi_D$ can be decomposed into elementary maps analogous to Mori's contractions of extremal rays. The birational maps between models $X_D$ as $D$ varies in the region can then be studied easily.    
\begin{setup} \label{setup0}
Let $\Delta_1, \dots, \Delta_r$ be big $\Q$-divisors on a projective $\Q$-factorial variety $Z$, and assume that $(Z, \Delta_i)$ is klt for all $i$. The \emph{adjoint  ring} $\mathfrak R= R(Z; K_Z+\Delta_1, \dots, K_Z+\Delta_n)$ is 
\[\mathfrak R= \bigoplus_{(n_1, \dots, n_r)\in \N^r} H^0(Z, n_1(K_Z+\Delta_1)+ \dots+n_r(K_Z+\Delta_r)),\]
where, for each $D\in \Div_\R(Z)$, 
$H^0(Z, D)= \{f\in k(Z)\mid \ddiv f+ D\geq 0\}$.
By the results of \cite{BCHM}, the \emph{support of $\mathfrak R$}
\[ 
\mcal C_\mathfrak R=  \{D\in\sum\R_+(K_Z+\Delta_i)\mid H^0(Z, D)\neq\{0\}\}\subseteq \Div_\R(Z)\]
is a rational polyhedral cone (see \cite[Theorem 3.2]{CaL10}). \end{setup}

As is shown in \cite{KKL12}, the convex geometry of $\mcal C_\mathfrak R$ determines
a \emph{geography} of ample models of divisors $K+\Delta$ as they vary in $\mcal C_\mathfrak R$. 
Since $\mcal C_\mathfrak R$ is spanned by klt adjoint divisors, an additional condition on the dimension of $\mcal C_\mathfrak R$ ensures that these ample models are $\Q$-factorial, so that they are the results of $(K_Z+\Delta)$-MMPs for $K_Z+\Delta\in \mcal C_\mathfrak R$. 
The results relevant to this paper are recalled in the following Theorem and Proposition.  
\begin{thm}\label{thm:2.1}
Let $Z$, $\Delta_1, \dots, \Delta_r$ and $\mathfrak R$ be as in Setup~\ref{setup0}. Assume that there is a big divisor $D\in \mcal C_\mathfrak R$. 
Then, there is a finite decomposition
\begin{equation*}\mcal C_\mathfrak R= \coprod \mcal A_i\end{equation*}
with the following properties.
\begin{enumerate}
\item[1.] each $\overline{\mcal A_i}$ is a rational polyhedral cone,\item[2.]
there is a rational map $\vphi_i\colon X\dashto X_i$, with $X_i$ is normal and projective, such that $\vphi_i$ is the ample model of all $D\in\mathcal A_i\cap \Div_\Q(Z)$,
\item[3.] if $\mathcal A_j\subseteq\overline{\mathcal A_i}$, there is a morphism $\varphi_{ij}\colon X_i\to X_j$ with $\vphi_i\simeq \vphi_{ij}\circ \vphi_j$,
\item[4.] if $\mathcal A_i$ contains a big divisor, $\varphi_i$ is a semiample model of all $D\in\overline{\mathcal{A}_i}\cap \Div_\Q(Z)$ and $X_i$ has rational singularities.
\end{enumerate}
\end{thm}\begin{proof}
See \cite[Theorems 4.2,4.5]{KKL12}.
\end{proof}
\begin{rem}\label{rem:2.5} I recall the properties of the decomposition of $\mcal C_\mathfrak R$ used in what follows. Since $\mathfrak R$ is finitely generated, \cite[Theorem 4.1]{ELMNP} implies: 
\begin{enumerate}
\item[(i)] $\mcal C_\mathfrak R$ is a closed rational polyhedral cone,
\item[(ii)] there is a positive integer $d$ and a resolution $f \colon \widetilde{Z} \lto Z$ such that $\Mob f^*(dD)$ is basepoint free for every $D\in \mcal C_\mathfrak R\cap \Div_\Z(Z)$, and $\Mob f^*(kdD)=k\Mob f^*(dD)$ for every $k\in \N$.
\item[(iii)] there is a finite rational polyhedral subdivision $\mcal C_\mathfrak R=\bigcup \mcal{C}_i$ into cones with disjoint interiors such that the restriction of $D\mapsto \Fix |f^*dD|$ to each $\mcal C_i$ is linear. 
\end{enumerate}
The cones $\mcal A_i$ in Theorem~\ref{thm:2.1} are the relative interiors of all cones that are intersections of the $\mcal C_i$s and their proper faces. For each $i$, $\vphi_i$ is induced by the Iitaka fibration $\psi_i\colon \widetilde{Z}\to \Proj R(\widetilde{Z}, \Mob f^*dD)$ for any $D\in \mcal A_i\cap \Div_\Z(Z)$.
\end{rem}

\begin{pro}\label{pro:2.1}Let $Z$ and $\mathfrak R$ be as in Theorem~\ref{thm:2.1}, and denote by $\pi\colon \Div_\R(Z) \lto N^1(Z)_\R$ the natural map.
\begin{enumerate}

\item[1.]
Let $L\subseteq \Div_\R(Z)$ be the maximal subspace with $A+L\subseteq \mcal C_\mathfrak R$ for some $\Q$-divisor $A$, 
and let $L^*\subseteq\Div_\R(Z)$ be a complementary subspace. Denote by $p\colon \Div_\R(Z) \to L^*$ the projection.
There is a finitely generated  ring ${\mathfrak R'}$ with support $\mcal C_{\mathfrak R'}= p(\mcal C_\mathfrak R)$. The cone  $\mcal C_{\mathfrak R'}$ contains no line and $\pi$ restricts to an isomorphism on $\mcal C_{\mathfrak R'}$. 
If $D\in \mcal C_\mathfrak R$ is a big $\Q$-divisor, $D$ and $p(D)$ have the same ample model.
\item[2.]

Let $\mcal C_\mathfrak R= \coprod \mcal A_i$ be the coarsest decomposition of Theorem~\ref{thm:2.1}. Then, $\mcal A_i= L+\mcal A_i'$, where $\mcal A'_i= p(\mcal A_i)$. The decomposition $\mcal C_{\mathfrak R'}= \coprod \mcal A'_i$ satisfies (1--4) in Theorem~\ref{thm:2.1}.
\end{enumerate}
Assume that $\dim \mcal C_{{\mathfrak R'}}= \rho(Z)$. Up to refining $\{\mcal A_i; i\in I\}$, we may assume that for all $i,j\in I$, every $\pi(\mcal A_i)$ is contained in one of the subspaces bounded by $H$, where $H\subset N^1(Z)_\R$ is any supporting hyperplane of $\pi(\mcal A_j)$.  
\begin{enumerate}\item[3.]
If $\dim \mcal A'_i= \rho(Z)$, $\vphi_i$ is the result of a $(K_Z+\Delta)$-MMP for all $(Z, \Delta)$ klt with $K_Z+\Delta\in \mcal A_i$. 
\item[4.]If $\mcal C_{\mathfrak R'}$ contains an ample divisor, $\vphi_i$ is the composition of a finite number of \emph{elementary contractions}, where an elementary contraction is an isomorphism in codimension $1$ or a morphism that contracts a single divisor. 
\end{enumerate}
\end{pro}
\begin{proof}
Let $L\subset \Div_\R(Z)$ be such that $A+L\subseteq \mcal C_\mathfrak R$ for some $\Q$-divisor $A$. If $D\in L\cap \Div_\Q(Z)$, then $D\equiv 0$, because both $A+nD$ and $A-nD$ are effective for all $n>\!>\!0$. By the results on convex sets recalled above, $L^*\cap \mcal C_\mathfrak R= p(\mcal C_\mathfrak R)$ is a cone that contains no line, and hence the cone over a rational polytope $\mcal P$. The vertices of $\mcal P$ are of the form $\varepsilon_i(K_Z+\Delta_i')\in \mcal C_\mathfrak R$ where $\varepsilon_i\in \Q_+$ and $\Delta'_i$ is a big $\Q$-divisor. Then, $p(\mcal C_\mathfrak R)= \mcal C_{\mathfrak R'}$, where ${\mathfrak R'}= R(Z; K_Z+\Delta'_1, \dots , K_Z+\Delta'_m)$ is finitely generated.  

If $D\in \mcal C_\mathfrak R\cap \Div_\Q(Z)$, since $p(D)\in \mcal C_\mathfrak R\cap \Div_\Q(Z)$ as well, both $R(Z, D)$ and $R(Z, p(D))$ are finitely generated. When $D$ is big, by Lemma~\ref{lem:2.1}, the ample models of $D$ and $p(D)$ coincide up to isomorphism because $D\equiv p(D)$.  

Let $f\colon \widetilde{Z} \to Z$ be a resolution as in Remark~\ref{rem:2.5}. Again by Lemma~\ref{lem:2.1}, if $D\in \mcal C_\mathfrak R$ is a big $\Q$-divisor, $\sB(D)= \sB(p(D))$, and $\Fix |f^*dD|= \Fix|f^*dp(D)|$. The coarsest decomposition of $\mcal C_\mathfrak R$ of \cite[Theorem 4.1]{ELMNP} depends only on regions of linearity of $D\mapsto \Fix |f^*dD|$ on $\{D\in \mcal C_\mathfrak R\mid D \mbox{ is big }\}$, so that $\mcal C_i= L+p(\mcal C_i)$, and $\mcal A_i= L+\mcal A_i'$.
The last two assertions are part of Theorems 4.4 and 5.4 in \cite{KKL12}. 
\end{proof}
\begin{nt}\label{nt:2.1}
Assume that $\mcal C_\mathfrak R$ is a rational polyhedral cone that contains no line. There is a hyperplane $H$ such that $\mcal C_\mathfrak R$ is the cone over $\mcal P_\mathfrak R= H\cap \mcal C_\mathfrak R$. 
The rational polytope $\mcal P_\mathfrak R$ is a \emph{base of $\mcal C_\mathfrak R$}; any two bases of $\mcal C_\mathfrak R$ are related by a smooth affine transformation.
When $\mcal C_\mathfrak R$ is as in Theorem~\ref{thm:2.1}, any $D\in \mcal P_\mathfrak R$ is a positive rational multiple of a klt adjoint divisor and 
\[\mcal P_\mathfrak R= \coprod(\mcal A_i\cap H)= \coprod \mcal Q_i\] 
where $\mcal Q_i$ is the relative interior of a rational polytope. This decomposition has the properties listed in Theorem~\ref{thm:2.1} and Proposition~\ref{pro:2.1}.  
\end{nt}

\begin{dfnlm}\label{dfnlm:2.1} Let $Z, \mathfrak R$ and $\{\mcal A_i; i\in I\}$ of $\mcal C_\mathfrak R$ be as in Proposition~\ref{pro:2.1}.  Assume that $\mcal C_\mathfrak R= \mcal C_{\mathfrak R'}$ contains no line. 

Denote by $\mcal C_\mathfrak R$ the polyhedral complex with $\Ob \mcal C_\mathfrak R= \{\overline{\mcal A_i}; i\in I\}$ and morphisms the inclusions of faces. The complex $\Ner (\mcal C_{\mathfrak R})$ is connected and: 
\begin{enumerate}
\item[1.]
There is a $1$-to-$1$ correspondence between vertices $\overline{\mcal A_i}^*$ of $\Ner \mcal C_\mathfrak R$ and the results $\vphi_i \colon Z \dashto X_i$ of $(K_Z+\Delta)$-MMPs that are ample models of $K_Z+\Delta\in \mcal C_\mathfrak R$. 
\item[2.]
Vertices $\overline{\mcal A_i}^*, \overline{\mcal A_j}^*$ of $\Ner \mcal C_\mathfrak R$ form an edge precisely when $\vphi_j \circ \vphi_i^{-1}$ or $\vphi_i\circ \vphi_j^{-1}$ is an elementary contraction.
\item[3.]
There is a $1$-to-$1$ correspondence between edge-paths on $\mcal N_\mathfrak R$ from $v_i$ to $v_j$ and decompositions of $X_i \dashto X_j$ into elementary contractions and inverses of elementary contractions dominated by $Z$.
\end{enumerate}
\end{dfnlm}
\begin{proof}
This is a straightforward reformulation of Proposition~\ref{pro:2.1}.\end{proof}
\begin{rem} \label{rem:2.6}Let $\mcal C_\mathfrak R= \bigcup_{1\leq j\leq N} \mcal C_j$ be the decomposition into cones of dimension $\rho(Z)$ associated to $\{\mcal A_i; i\in I\}$.
Edge paths on $\Ner \mcal C_\mathfrak R$ are dual to piecewise linear paths between points in $\inte (\mcal C_j)$ that cross all $\partial \mcal C_j$ transversally along codimension $1$ faces.
\end{rem}
\section{Complexes and the Sarkisov Program}
\label{constructions}
Let $\{X_j/S_j, \Phi_{j,j'}\}$ be a collection of Mori fibre spaces and birational maps between them. 
As mentioned in Section~\ref{models}, there are divisors $\Theta_j$ on a nonsingular variety $Z$ such that each $X_j/S_j$ is the result of a $K_Z$-MMP and of a $(K_Z+\Theta_j)$-MMP, where $K_Z+\Theta_j$ is strictly pseudoeffective. 
In Section~\ref{subsec:3.1}, I show that $\{\Theta_j\}_j$ can be chosen so that all $K_Z+\Theta_j$ lie in a good region $\mcal C_\mathfrak R\subset \Div_\R (Z)$ in the sense of Section~\ref{geog}. As above, I denote by $\mcal C_\mathfrak R$ both the cone and the polyhedral complex supported on this cone and defined by a decomposition as in Proposition~\ref{pro:2.1}.
Any $\Phi_{j,j'}$ can then be decomposed into a finite number of steps that are elementary contractions or inverses of elementary contractions dominated by $Z$. By Definition-Lemma~\ref{dfnlm:2.1}, a decomposition is determined by considering suitable paths in $\mcal C_{\mathfrak R}$.  In general, some of the models (of divisors $D\in \mcal C_\mathfrak R$) that appear in such a decomposition are not Mori fibre spaces. 
To achieve a decomposition of $\Phi_{j,j'}$ where all elementary steps are birational maps between Mori fibre spaces, I consider paths on $\mcal C_\mathfrak R$ that lie on the part of $\partial \mcal C_\mathfrak R$ consisting of strictly pseudoeffective divisors; this is done in Section~\ref{subsec:3.2}. 
Last, in Section~\ref{subsec:3.3}, given a cone $\overline {\mcal A}$ in the decomposition of $\mcal C_\mathfrak R$ with  $\mcal A\cap \B(Z)=\emptyset$, I define a \emph{residual ring} $\resi_\mcal A\mathfrak R$. If $\vphi \colon Z \lto W$ is the ample model corresponding to $\mcal A$, the support of $\resi_\mcal A \mathfrak R$ inherits a geography that is a \emph{geography of models over $W$}. 
\subsection{Cones of divisors for the Sarkisov Program}
\label{subsec:3.1}
Consider a collection $\{X_1/S_1, \dots, X_r/S_r; \Phi_{j,j'}\colon X_j\dashto X_{j'}\}$ of Mori fibre spaces and birational maps between them. Denote by $p_j\colon X_j\to S_j$ the fibration morphisms. 
The following construction is similar to the one in \cite[Section 4]{HM09}.

\begin{pro}\label{pro:3.1}Let $f=(f_1, \dots, f_r)\colon Z\to X_1\times \dots \times X_r$ be a  resolution of all maps $\Phi_{j,j'}$. There are ample $\Q$-divisors $\Delta_1, \dots , \Delta_n$ such that the rational polyhedral cone $\mcal C_\mathfrak R$ associated to $\mathfrak{R}= R(Z; K_Z+\Delta_1, \dots, K_Z+ \Delta_n)$ contains an ample divisor and has the following properties.  

Let $V=\sum\R(K_Z+\Delta_i)$, $\partial^+\mcal C_\mathfrak R= \{D\in \mcal C_\mathfrak R|D\not \in \B(Z)\}$ and $\mcal C_\mathfrak R= \coprod_{i\in I} \mcal A_i$ the coarsest decomposition in Theorem~\ref{thm:2.1}.
Let $\vphi_i\colon Z \dashto Z_i$ be the ample model of any $D\in \mcal A_i\cap \Div_\Q(Z)$. Then:
\begin{enumerate}
\item[1.] $\mcal C_\mathfrak R$ is a cone of dimension $\rho(Z)$ that contains no line,
\item[2.] there are indices $j_0, j_1\in I$ such that  $f_j\simeq \vphi_{j_0}$ and $ p_j\circ f_j\simeq \vphi_{j_1}$. Moreover, $\mcal A_{j_1}\subset \overline{\mcal A_{j_0}}$ and $\dim \mcal A_{j_0}= \rho(Z)= \dim \mcal A_{j_1}+1$,
\item[3.] for every $\Phi_{j,j'}$ in the collection, both $\vphi_{j_0}$ and $\vphi_{j'_0}$ are results of log-MMPs for $(Z, \Delta)$ klt with $K_Z+\Delta \in V$. Further, $\vphi_{j_0}$ and $\vphi_{j'_0}$ are MMPs with scaling by suitable ample divisors $A,A'\in V$. 
\item[4.] $\partial^+\mcal C_\mathfrak R$ is purely $\rho(Z)-1$ dimensional,
\item[5.] if $\partial^+\mcal C_\mathfrak R\subseteq\bigcup_{j=1}^r \overline{\mcal A_{j_1}}$, any result of a log-MMP for $(Z, \Delta)$ klt with $K_Z+\Delta\in V$ non-pseudoeffective that is an MMP with scaling by an ample divisor $A\in V$ is one of the $X_j/S_j$ in the collection. 
\end{enumerate}
\end{pro}
\begin{proof}
Observe that it is enough to construct a finitely generated ring $\mathfrak R= R(Z; K_Z+\Delta_1, \dots, K_Z+\Delta_n)$ such that (2--5) hold and such that the numerical classes of elements of $\mcal C_\mathfrak R$ span $N^1(Z)_\R$. Indeed, if $L\subset V$ is the maximal subspace with $\mcal C_\mathfrak R+L\subseteq \mcal C_\mathfrak R$, then $p(\mcal C_\mathfrak R)=\mcal C_{\mathfrak R'}$ constructed as in Proposition~\ref{pro:2.1} satisfies (1--5). 

\emph{Step 1.} I first construct a ring $\mathfrak R$ satisfying (2) and such that the classes of divisors in $\mcal C_\mathfrak R$ span $N^1(Z)_\R$.

Let $f=(f_1, \dots, f_r)\colon Z\to X_1\times \dots \times X_r$ be a common resolution that resolves all maps $\Phi_{j,j'}$, and let $\{E_k\}$ be the collection of $f$-exceptional divisors. Since $X_j$ is terminal, 
\begin{equation}\label{eq:3.1} K_Z= f_j^* K_{X_j}+ G_j,\end{equation}
where $G_j>0$ is a $\Q$-divisor whose support contains all $f_j$-exceptional divisors. 
Fix a rational number $\lambda>1$. 
For each $j$, let $\{A^l_{j}\}$ be ample $\Q$-divisors on $S_j$ whose numerical classes span $N^1(S_j)_\R$ and such that ${-}(\lambda-1)K_{X_j} + p_j^*A^l_j$,  ${-}\lambda K_{X_j} + p_j^*A_j^l$, and ${-}K_{X_j} + p_j^*A^l_j$
are ample. 
Denote by $E^{j}= \sum E_k$ the reduced sum of $f_j$-exceptional divisors, and by $E^j_k= E^j+ E_k$, where $E_k$ ranges over $f_j$-exceptional divisors. Let $0<\delta<\!\!<1$ be a rational number such that $G_{j}- \delta E_j^k>0$ and its support contains all $f_j$-exceptional divisors.
Define a collection of divisors 
\begin{eqnarray}\label{eq1.1}
\Theta^{k,l}_{j}= f_j^*\{{-}\lambda K_{X_j}+ p_j^*A_j^l\}-\delta E_j^k \mbox{ and } D^{k,l}_j= f_j^*\{{-}K_{X_j} +  p_j^* A_j^l\}-\delta E_j^k,\nonumber
\end{eqnarray}
 we may assume that $\delta$ is chosen so that $\Theta^{k,l}_j$, and $D_{j}^{k,l}$ are ample.
Then, 
$$K_Z+\Theta^{k,l}_j= f_{j}^*\{{-}(\lambda -1)K_{X_j}+ p_j^*A_j^l\}+G_j -\delta E^j_k$$
is big, and $f_j$ is $K_Z+\Theta^{k,l}_j$-negative, and is the ample model of $K_Z+\Theta^{k,l}_j$.
Similarly, since
$$K_Z+D^{k,l}_j= f_j^*\{ p_j^*A_j^l\}+G_j-\delta E^j_k,$$
$f_j$ is a semiample model for $K_Z+D^{k,l}_j$, and $p_j \circ  f_j$ is the ample model of $K_Z+D_j^{k,l}$. Since $D_j^{k,l}$ is ample and $K_Z+D_j^{k,l}$ is not big, $K_Z$ is not pseudoeffective, and $Z\dashto X_j/S_j$ is the result of a $K_Z$-MMP with scaling by any $D_{j}^{k,l}$.

Let $\Delta_i$ be the divisors in $\{\Theta_j^{k,l}, D_j^{k,l}\}_{j,k,l}$. If $\sum \R_+(K_Z+\Delta_i)$ does not contain an ample divisor, add to the collection $\{\Delta_i\}$ a very general ample divisor $\Delta$ such that $K_Z+\Delta$ is ample.  

Since each $\Delta_i$ is ample, up to replacing $\Delta_i$ with $\varepsilon_i H_i\sim_\Q\Delta_i$, where $\varepsilon_i\in \Q$, $0<\varepsilon_i<\!<\! 1$ and $H_i$ is a very ample $\Q$-divisor, we may assume that the pairs $(Z, \Delta_i)$ are klt.
The ring $\mathfrak R= R(Z; K_Z+\Delta_1, \dots , K_Z+\Delta_n)$  is finitely generated by \cite[Theorem 3.2]{CaL10}. 
Since $\mcal C_\mathfrak R$ contains an ample divisor, there is a coarsest decomposition $\mcal C_\mathfrak R= \coprod_{i\in I} \mcal A_i$ as in Theorem~\ref{thm:2.1}. By uniqueness of ample models, for each $j$, there are indices $j_0, j_1\in I$ with $f_j= \vphi_{j_0}$ and $K_Z+\Theta_{j}^{k,l}\in \mcal A_{j_0}$, and $p_j\circ f_j= \vphi_{j_1}$ and $K_Z+D_{j}^{k,l}\in \mcal A_{j_1}$  for all $k,l$. The divisors $K_Z+D_j^{k,l}$ are not big, and $\mcal A_{j_1}\subseteq \partial^+\mcal C_\mathfrak R$, and $\vphi_{j_1}= p_j\circ \vphi_{j_0}$. Also, since $\lambda$ can be chosen arbitrarily close to $1$ in the definition of divisors of the form of $\Theta_{j}^{k,l}$,  $\mcal A_{j_1}\cap  \overline{\mcal A_{j_0}}\neq \emptyset$, and by definition of the decomposition \eqref{eq:3.2}, $\mcal A_{j_1}\subsetneq \overline{\mcal A_{j_0}}$ . The only thing that is left to prove is that the numerical classes of divisors in $\mcal A_{j_0}$ span $N^1(Z)_\R$.  
 
The classes $[K_Z+\Theta_j^{k,l}]$ (resp.~$[K_Z+D_j^{k,l}]$) span $N^1(Z)_\R$ (resp.~a subspace of codimension $1$) because $f_j$ is a resolution and $p_j$ a Mori fibration, so that
\[ N^1(Z)_\R= (p_j\circ f_j)^* N^1(S_j)_\R \oplus \R [f_j^*({-}K_{X_j})] \oplus \bigoplus \R[E_k]\] 
where the sum runs over $f_j$-exceptional divisors, and $\bigoplus \R[E_k]= \bigoplus \R[E_k^j]$.   

\emph{Step 2.} 
For (3), let $\Phi_{j,j'}\colon X_j\dashto X_{j'}$ be a birational map in the collection. As in Step 1, denote by $j_0, j_1\in I$ (resp.~$j'_0, j'_1$), the indices such that $f_j$ and $p_j\circ f_j$ (resp.~$f_{j'}$ and  $p_{j'}\circ f_{j'}$) are the ample models associated to $\mcal A_{j_0}, \mcal A_{j_1}$ (resp.~to $\mcal A_{j_0}, \mcal A_{j'_1})$. 
As noted in Step 1, $K_Z$ is not pseudoeffective, and 
there are ample divisors $D_j, D_{j'}$ with $K_Z+D_j\in \mcal A_{j_1}$ and $K_Z+D_j'\in \mcal A_{j'_1}$.
 Then, for $0\leq \varepsilon<\!<\! 1$, if $\Delta= \varepsilon(D_j+ D_{j'})$, $(Z, \Delta)$ is klt and $K_Z+\Delta \not \in \mcal C_\mathfrak R$ and $f_j$ (resp.~$f_{j'}$) is an MMP with scaling of $(Z, \Delta)$ by $\R_+D_j$ (resp.~$\R_+D_{j'}$).  


\emph{Step 3.}
I now prove the statements on $\partial^+\mcal C_\mathfrak R$. Using the notation set in \ref{nt:2.1}, let $\mcal P_\mathfrak R$ be a base of $\mcal C_\mathfrak R$. Since $D$ is big if and only if $\lambda D$ is big for all $\lambda >0$, to prove (4), it is enough to ensure that $\partial^+\mcal P _\mathfrak R=\{D\in \mcal P_\mathfrak R\mid D \mbox{ is not big }\}$ is purely $(\rho(Z)-2)$-dimensional. 
Assume that this is not the case, and let $\partial^+\mcal P_\mathfrak R= T\coprod T'$, where $\dim T= \rho(Z)-2>\dim T'$. 
Note that $T\neq \emptyset$ because for any $X_j/S_j$ in the collection, if $\mcal A_{j_1}$ is defined as above, $\mcal A_{j_1}\cap \mcal P_{\mathfrak R}\subseteq \partial^+ \mcal P_\mathfrak R$ is $(\rho(Z)-2)$-dimensional. 
Let $\mcal P'_\mathfrak R\subsetneq \mcal P_\mathfrak R$ be a polytope of the same dimension as $\mcal P_\mathfrak R$ that contains all the vertices of $\mcal P_\mathfrak R$ in $T$ but none of those in $T'$. The vertices of $\mcal P'_\mathfrak R$ are of the form $\varepsilon_1(K_Z+\Delta'_1), \dots , \varepsilon_m(K_Z+\Delta'_m)$  for $\varepsilon_i\in \Q$, and $(Z, \Delta'_i)$ klt with $\Delta'_i$ big. Then, $\mathfrak R'= R(Z, K_Z+\Delta'_1, \dots, K_Z+\Delta'_m)$ is finitely generated, and $\mcal C_\mathfrak R'$ is the cone spanned by $\mcal P'_\mathfrak R$. Further, $\mathfrak R'$ satisfies (1--3), because for all $X_j/S_j$, $\mcal A_{j_0}\cap \mcal P'_{\mathfrak R}$ is $(\rho(Z)-1)$-dimensional and $\mcal A_{j_1}\cap \mcal P'_{\mathfrak R}$ is $(\rho(Z)-2)$-dimensional. After replacing $\mathfrak R$ by $\mathfrak R'$, (4) holds.    

For (5), let $(Z, \Delta)$ be a klt pair with $K_Z+\Delta\in V$ not pseudoeffective. Let $\Theta$ be such that $K_Z+\Theta\in \mcal C_\mathfrak R$ is not big, $(Z, \Theta)$ is klt and $\Theta-\Delta\in V$ is ample. Let $\mcal A_{i_0}$ be the cone of maximal dimension with $K_Z+\Theta\in \overline{\mcal A_{i_0}}$. Since $K_Z+\Theta$ is not big, $K_Z+\Theta$ belongs to a codimension $1$ face $\overline{\mcal A_{i_1}}$ of $\overline{\mcal A_{i_0}}$ with $K_Z+\Theta\in \overline{\mcal A_{i_1}}$.  
In the notation of Theorem~\ref{thm:2.1}, $\vphi_{i_0, i_1}\colon Z_{i_0}\to Z_{i_1}$ is a Mori fibre space that is the result of a $(K_Z+\Delta)$-MMP.  By Lemma~\ref{lem:2.2}, ${-}(K_{Z_{i_0}}+\vphi_{i_0}\Delta)$ is $\vphi_{i_0, i_1}$-nef,  and since $\rho(Z_{i_0}/Z_{i_1})=1$, it is $\vphi_{i_0,i_1}$-ample because $\Theta-\Delta$ is ample and $(\vphi_{i_0})_*(\Theta- \Delta)\not \equiv 0$. 
If the $\overline{\mcal A_{j_1}}$ cover $\partial^+ \mcal C_\mathfrak R$, $\mcal A_{i_1}= \mcal A_{j_1}$ for some $j$, and $X_j/S_j\simeq Z_{i_0}/Z_{i_1}$. 
\end{proof}
As a by-product, the proof of Proposition~\ref{pro:3.1} shows the next Lemma.
\begin{lem}\label{lem:3.1}
Let $\mcal A_{i_0}, \mcal A_{i_1}$ be cones in the decomposition of $\mcal C_\mathfrak R$ such that $\mcal A_{i_1}$ has maximal dimension, $\mcal A_{i_1}\subset \overline{\mcal A_{i_0}}$ has codimension $1$ in $\mcal A_{i_1}$ and contains no big divisor. Then $(Z_{i_0}, (\vphi_{i_0})_*\Delta)/Z_{i_1}$ is a log Mori fibre space, where $(Z, \Delta)$ is any klt pair with $\Delta= \Theta-H$ for $H$ ample and $K_Z+\Theta\in \mcal A_{i_1}\cap \Div_\Q(Z)$. 
\end{lem}
\begin{rem}\label{rem:3.1} In Proposition~\ref{pro:3.1}, $K_Z$ is not effective and all $Z_{i_0}/Z_{i_1}$ that are results of $K_Z$-MMPs with scaling by nef divisors $\Theta$ with $K_Z+\Theta\in V$ are Mori fibre spaces.
\end{rem}
\begin{rem}\label{rem:3.2}
By construction, the movable part of any $D\in\sum  \R_+(K_Z+\Delta_i)$ is basepoint free, hence, in the notation of Theorem~\ref{thm:2.1}, $f$ is an isomorphism and all $\vphi_i$ are morphisms.  
\end{rem}
\begin{rem}\label{rem:3.7}
If $\mcal A_{i_1}\cap \B(Z)= \emptyset$ and $\mcal A_{i_1}\subseteq \overline{\mcal A_{i_0}}$ with $\dim \mcal A_{i_0}= \rho(Z)$, then $\vphi_{i_0, i_1}\colon Z_{i_0}\to Z_{i_1}$ has $1\leq \rho(Z_{i_0}/Z_{i_1})\leq \dim \mcal A_{i_0}- \dim \mcal A_{i_1}$.
\end{rem}
\subsection{Polyhedral complexes}
\label{subsec:3.2}
In this section, I define a polyhedral complex supported on $\partial^+ \mcal C_\mathfrak R$, where  $\mcal C_\mathfrak R$ is as in Proposition~\ref{pro:3.1}.
\begin{nt}\label{nt}
Let $\{X_1/S_1, \dots, X_r/S_r; \Phi_{j,j'}\}$ be a collection of Mori fibre spaces and birational maps between them and $\mathfrak R= R(Z; K_Z+\Delta_1, \dots, K_Z+\Delta_n)$ the ring constructed in Proposition~\ref{pro:3.1}. Denote by $V= \sum \R(K_Z+D_i)$. Fix the coarsest decomposition 
\begin{equation} \label{eq:3.2} \mcal C_\mathfrak R= \coprod_{i\in I} \mcal A_i
\end{equation}
as in Theorem~\ref{thm:2.1}. As above, $\vphi_i\colon Z\lto X_i$ is the ample model of all $D\in \mcal A_i$. 
Let $\mcal P_\mathfrak R$ be a base of $\mcal C_\mathfrak R$ and $\mcal Q_i= \mcal A_i\cap \mcal P_{\mathfrak R}$ the induced decomposition as in Notation~\ref{nt:2.1}. 

Enlarging the collection if necessary, assume that $X_j/S_j$ exhaust the possible results of $(K_Z+\Delta)$-MMPs with scaling by ample divisors $A\in V$ for non-effective $K_Z+\Delta \in V$.  
Let $\mcal B^+(I)=\{i\in I| \mcal A_i\cap \B(Z)= \emptyset\}$, where $I$ is the index set in \eqref{eq:3.2}, so that
\begin{equation}\label{eq:3.3}
\partial^+ \mcal C_\mathfrak R= \coprod_{i\in \mcal B^+(I)} \mcal A_i.
\end{equation}
By construction, $\partial^+\mcal C_\mathfrak R$ is purely $(\rho(Z)-1)$-dimensional and is the cone over $\partial^+\mcal P_\mathfrak R= \coprod_{i\in \mcal B^+(I)} \mcal Q_i$, where $\mcal Q_i$ are defined in Notation~\ref{nt:2.1}. 
\end{nt}
\begin{nt}\label{nt:3.2}
Let $Z$ and $\mathfrak R$ be as in \ref{nt}, and let $\mcal C_\mathfrak R$ be the polyhedral complex  of Definition-Lemma~\ref{dfnlm:2.1}. 
Let $\mcal B^+(\mcal C_R)$ be the subcomplex of $\mcal C_{\mathfrak R}$ with objects $\overline{\mcal A_i}$ for $i\in \mcal B^+(I)$; $\mcal B^+(\mcal C_{\mathfrak R})$ is a pure polyhedral complex of dimension $\rho(Z)-1$. 
Similarly, the polyhedral complex $\mcal B^+(\mcal Q_\mathfrak R)$ with objects $\overline{\mcal Q_i}$ for $i\in \mcal B^+(I)$ is pure of dimension $\rho(Z)-2$. 
Denote by $\mcal N_\mathfrak R$ the simplicial complex $\Ner \mcal B^+(\mcal Q_\mathfrak R)$.   
\end{nt}
\begin{rem}There is a natural identification between $\mcal N_\mathfrak R$ and the $(\rho(Z)-2)$-skeleton of $\Ner \mcal B^+(\mcal C_\mathfrak R)$. 
Any $\rho(Z)$ distinct vertices of $\Ner \mcal B^+(\mcal C_\mathfrak R)$ span a $(\rho(Z)-1)$-simplex dual to the cone $\{0\}$. \end{rem}

 \begin{rem}\label{rem:3.6} By definition, the geometric realisations of the complexes defined above are $|\mcal C_\mathfrak R|= \mcal C_\mathfrak R$, $|\mcal B^+(\mcal C_\mathfrak R)|= \partial^+\mcal C_\mathfrak R$ and $|\mcal B^+(\mcal Q_\mathfrak R)|= \partial^+ \mcal P_\mathfrak R$. 
 \end{rem}
In what follows, I always assume that $\rho(Z)\geq 3$, as the Sarkisov Program reduces to the $2$-ray game otherwise. In particular, the vertices and edges of $\mcal N_\mathfrak R$ are precisely the vertices and edges of $\Ner \mcal B^+(\mcal C_\mathfrak R)$. I use both constructions of $\mcal N_\mathfrak R$.  
The next proposition states some easy properties of  $\mcal C_\mathfrak R, \mcal B^+(\mcal C_\mathfrak R)$, and $\mcal N_\mathfrak R$.
\begin{pro}\label{pro:3.2}
If $\overline{\mcal A_i},\overline{\mcal A_j}$ are distinct facets of $\mcal C_\mathfrak R$, then $\vphi_i\not \simeq \vphi_j$.
If $\overline{\mcal A}$ is a facet of $\mcal B^+ (\mcal C_\mathfrak R)$, there is a unique facet $\overline{\mcal A'}$ of $\mcal C_\mathfrak R$ with $\mcal A\subseteq \overline{\mcal A'}$.

There is a $1$-to-$1$ correspondence between the vertices of $\mcal N_\mathfrak R$ and the Mori fibre spaces produced by MMPs with scaling by nef divisors $H\in V$. 
\end{pro}
\begin{proof} 
Let $\overline{\mcal A_i}$ and $\overline{\mcal A_j}$ be distinct facets of $\mcal C_\mathfrak R$. If $\vphi_i$ is the ample model for all $D\in \mcal A_i\cup \mcal A_j$, then $D\mapsto \Fix|dD|$ is linear on $\overline{\mcal A_i\cup \mcal A_j}$. This is impossible because the decomposition \eqref{eq:3.2} is assumed to be coarsest. 

The second assertion follows immediately from the definition of the decomposition \eqref{eq:3.2}. If $\mcal A_i, \mcal A_j$ are facets of $\mcal C_{\mathfrak R}$ with $\overline{\mcal A_{i,j}}=\overline{\mcal A_i}\cap \overline{\mcal A_j}$ and $\dim \mcal A_{i,j}= \rho(Z)-1$, any $D\in \mcal A_{i,j}$ can be written $D= D_i+D_j$ for $D_i\in \mcal A_i$ and $D_j\in \mcal A_j$. But then, $D$ is big because $D_i$ and $D_j$ both are, and $\mcal A_{i,j}\not \in \mcal B^+(\mcal C_\mathfrak R)$. In particular, a facet of $\mcal B^+(\mcal C_\mathfrak R)$ is not the intersection of two facets of $\mcal C_\mathfrak R$.

Last, let $v_i= (\overline{\mcal A_{i_1}})^*$ be a vertex of $\mcal N_\mathfrak R$, and $\overline{\mcal A_{i_0}}$ the facet of $\mcal C_\mathfrak R$ that contains $\mcal A_{i_1}$.
By Lemma~\ref{lem:3.1} and Remark~\ref{rem:3.1}, $Z_{i_0}/Z_{i_1}$ is a Mori fibre space that is the result of a $K_Z$-MMP by scaling in $V$, and all results of $K_Z$-MMPs with scaling in $V$ are of this form. This finishes the proof.
\end{proof}
\begin{rem}\label{rem:3.3}
If $\overline{\mcal A_i}, \overline{\mcal A_j}$ are distinct facets of $\mcal{B}^+(\mcal C_\mathfrak R)$, the associated ample models $\vphi_i$ and $\vphi_j$ need not be distinct. When $\vphi_i\simeq \vphi_j$, $\vphi_i$ has distinct factorisations through semiample models that are ample models for the (not necessarily distinct) facets $\mcal A_i', \mcal A_j'$, compare with Corollary~\ref{cor:3.1}. 
\end{rem} 
 
\begin{exa}\label{exa:4.1} Let $S$ be the blow up of $\PS^2$ in $2$ points $P_1\neq P_2$. Let $E_1,E_2$ be the $(-1)$-curves lying over $P_1, P_2$ and $L$ the proper transform of the line through $P_1$ and $P_2$.
Figure~\ref{fig:1.1} represents a base $\mcal P_{\mathfrak R}$ and the decomposition induced by \eqref{eq:3.2}, where $\mathfrak R= R(S, K_S+A+E_1, K_S+A+E_2, K_S+A+L)$ for an ample divisor $A\sim_{\Q}{-}K_S$. Then, $\mcal C_\mathfrak R=\Effb(S)= \R_+[L]+\R_+[E_1]+\R_+[E_2]$. 
If $\mcal A_i$ denotes the interior of the cones $\mcal C_i$ in Figure~\ref{fig:1.1}, the ample models associated to $\mcal A_0, \dots , \mcal A_4$ are $\vphi_0\colon S\stackrel{\simeq}\lto S$, $\vphi_1\colon S\to \PS^1_a\times \PS^1_b$, $\vphi_2\colon S\to  \F_{1a}$, $\vphi_3\colon S\lto \F_{1b}$ and $\vphi_4\colon S\lto \PS^2$. 
The complex $\mcal N_{\mathfrak R}$ is the graph with $5$ vertices and $5$ edges:
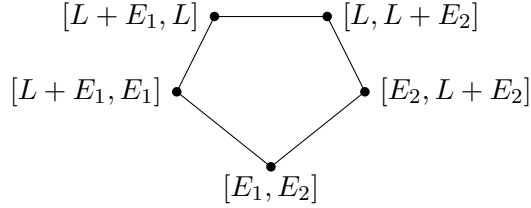
\begin{center}
\begin{pspicture}[psscale=.5](-0.5,-0.5)(2.5,2.5)
\psdots(1.25,0)(0,1)(2.5,1)(0.5,2)(2, 2)
\pspolygon[linewidth=0.4pt](1.25,0)(0,1)(0.5,2)(2, 2)(2.5,1)
\rput[t](1.25, -0.1){$[E_1,E_2]$}\rput[r](-0.2,1){$[L+E_1,E_1]$}\rput[r](0.3,2){$[L+E_1,L]$}\rput[l](2.2,2){$[L, L+E_2]$}\rput[l](2.7, 1){$[E_2, L+E_2]$}
\end{pspicture}
\end{center}
\end{exa}

\begin{dfn}\label{dfn:3.2} 
A birational map $\Phi\colon Z_i\dashto Z_j$ between the ample models associated to $\mcal A_i, \mcal A_j$ in \eqref{eq:3.2} is \emph{dominated by $Z$} if $Z$ is a resolution of $\Phi$, that is if $\Phi\simeq \vphi_{j}\circ \vphi_i^{-1}$. Equivalently, $\Phi$ is dominated by $Z$ when $\Phi\circ \vphi_i$ is the ample model of some nonzero $\Q$-divisor $D\in \mcal A_j$.  

An elementary Sarkisov link $L_{i,j}\colon X_i/S_i\dashto X_{j}/S_{j}$ as in \eqref{esl} is \emph{dominated by $Z$ in $V$} if $X_i/S_i$ and $X_j/S_j$ are the results of $(K_Z+\Delta)$-MMPs with scaling in $V$ for $(Z, \Delta)$ klt and $K_Z+\Delta\in V$. \end{dfn}
\begin{rem}
The map $\Phi\colon Z_i \dashto Z_j$ is dominated by $Z$ when $\Phi$ is induced by the geography of ample models associated to the decomposition \eqref{eq:3.2} of $\mcal C_\mathfrak R$. When $\Phi= L_{i,j}$ is a Sarkisov link, this is equivalent to requiring that $X_i\dashto X_j$ is dominated by $Z$ and that the induced map $Z\dashto T_{i,j}$ is the ample model for some nonzero $\Q$-divisor $D\in \mcal C_\mathfrak R$.\end{rem}

\subsection{Residual rings and complexes}
\label{subsec:3.3}
In this section, I define subrings of $\mathfrak R$ associated to affine subspaces $H\subset V$, and show that, for suitable choices of $H$, the geography of ample model induced by \eqref{eq:3.2} on the support of these subrings corresponds to running suitable relative Minimal Model Programs.  

\begin{dfn}\label{dfn:3.3}
An affine subspace $H=D+W\subset \Div_{\R}(Z)$ of dimension $d\leq \rho(Z)-1$ is in \emph{general position} with respect to $\mcal C_\mathfrak R$ if $D$ is a nonzero $\Q$-divisor and if $H \cap \mcal C_\mathfrak R$ is not contained in any union or intersection of cones $\mathcal{A}_i$ in \eqref{eq:3.2}. 

Let $\mcal A$ be a cone of dimension $d$ in \eqref{eq:3.3}. An affine subspace $H$ of dimension $\rho(Z)-1-d$ is in general position with respect to $\mcal C_\mathfrak R$ and $\mcal A$ if it is in general position with respect to $\mcal C_\mathfrak R$ and if $H\cap \mcal A\neq \emptyset$. 
\end{dfn}

The following two lemmas are reformulations of results in \cite{HM09}.
\begin{lem} \label{lem:3.2}Let $D\in V$ be a nonzero adjoint $\Q$-divisor.  
For any $k\leq \rho(Z)-1$ the subset $U\subset Gr(k,V)$ of $k$-dimensional affine subspaces $D+W\subset V$ in general position with respect to $\mcal C_\mathfrak R$ is Zariski dense. 

\end{lem}
\begin{proof}This is an immediate consequence of the existence of the finite decomposition \eqref{eq:3.2} into rational polyhedral cones. \end{proof}
\begin{lem}\label{lem:3.3}
There is a $1$-to-$1$ correspondence between edges of $\mcal N_\mathfrak R$ and elementary Sarkisov links dominated by $Z$.
\end{lem}
\begin{proof} 
I use the notation of the proof of Proposition~\ref{pro:3.1}. 
First, I show that an edge of $\mcal N_\mathfrak R$ determines an elementary link dominated by $Z$.
Let $v_i, v_j$ be vertices of $\mcal N_\mathfrak R$ that span a $1$-simplex. Let $i_1,  j_1\in \mcal B^+(I)$ be such that $\overline{\mcal A_{i_1}}= v_i^*$ and $\overline{\mcal A_{j_1}}= v_j^*$, and $i_0, j_0\in I$ the indices of the facets of $\mcal C_\mathfrak R$ with $\mcal A_{i_1}\subseteq \overline{\mcal A_{i_0}}$ and $\mcal A_{j_1}\subseteq \overline{\mcal A_{j_0}}$.  Denote by $X_i/S_i$ (resp.~$X_j/S_j$) the Mori fibre spaces $Z_{i_0}/Z_{i_1}$ (resp.~$Z_{j_0}/Z_{j_1}$).

Let $\overline{\mcal A}=\overline{\mcal A_k}$ be the object of $\mcal B^+(\mcal C_\mathfrak R)$ dual to the $1$-simplex $\{v_i, v_j\}$ of $\mcal N_\mathfrak R$. Then, $\overline{\mcal A}= \overline{\mcal A_{i_1}}\cap \overline{\mcal A_{j_1}}$ and $\mcal A$ has dimension $\rho(Z)-2$. Let $\vphi\colon Z\lto T_{i,j}$ be the ample model associated to the face $\mcal A$. 
Then, by Theorem~\ref{thm:2.1}(3), there are morphisms $\vphi_{i_0,k}$ and $\vphi_{j_0,k}$ that decompose as 
\[ X_i \stackrel{\vphi_{i_0, i_1}}\lto S_{i} \stackrel{\vphi_{i_1,k}}\lto T_{i,j} \mbox{ and } X_i \stackrel{\vphi_{j_0, j_1}}\lto S_{j} \stackrel{\vphi_{j_1,k}}\lto T_{i,j} \]
such that $\vphi\simeq \vphi_{i_0,k}\circ \vphi_{i_0}\simeq \vphi_{j_0,k}\circ \vphi_{j_0}$.
 
Fix $D\in \mcal A\cap \Div_\Q(Z)$, and let $H=D+W$ be a $2$-dimensional affine subspace in general position with respect to $\mcal C_\mathfrak R$ and $\mcal A$. Then,  $\mcal C_\mathfrak R\cap H= \mcal P_\mathfrak R\cap H$ for a suitable base $\mcal P_\mathfrak R$, and $\mcal P_\mathfrak R\cap H$ is as in Figure 2, where $\mcal Q_i= \mcal A_i\cap H$. 
\begin{figure}[h]\begin{center}
\label{fig:2}
\begin{pspicture}(-4,-0.3)(4, 1.2)
\psdot (0,0)
\rput[t](0,-0.1){$D$}
\rput[t](-1.5, -0.1){$\mcal Q_i$}\rput[t](1.5, .1){$\mcal Q_j$}\rput(0.1,.4){$\cdots$}
\rput[b](-2.7, 0.1){$\mcal Q_{i_0}$}\rput[b](2.7, 0.45){$\mcal Q_{j_0}$}
\psline[linewidth=0.4pt](0,0)(-4,0)\psline[linewidth=0.4pt](0,0)(-3.2,1) \psline[linewidth=0.4pt](0,0)(-2.5, 1)\psline[linewidth=0.4pt](0,0)(-1.5,1)\psline[linewidth=0.4pt](0,0)(3, 1)\psline[linewidth=0.4pt](0,0)(4,0.6)
\end{pspicture}
\vspace{-10pt}
\caption{}
\end{center}
\vspace{-10pt}
\end{figure}
Since $D$ is an exposed point of $\mcal P_\mathfrak R\cap H$ and since the decomposition \eqref{eq:3.2} is finite, there is a line $L$ such that $(\mcal P_\mathfrak R\cap H)\cap L$ contains no other vertices of $\mcal P_{\mathfrak R}\cap H$, and the only $2$-dimensional polytopes $\overline{\mcal Q_i}$ that $L$ intersects contain $D$ (e.g.~a small translation of the supporting hyperplane of $\mcal P_\mathfrak R\cap H$ at $D$).
Let $\mcal A_1, \cdots, \mcal A_n$ be the facets of $\mcal C_\mathfrak R$ that have $1$-dimensional intersection with $L$, where $\mcal A_i\cap H= \mcal Q_i$. Then, $\mcal A \subset \overline{\mcal A_l}$ for $l=1, \dots, n$, there is a morphism $Z_l\to T_{i,j}$ for all $l$, with $\rho(Z_l/T_{i,j})\leq 2$ by Remark~\ref{rem:3.7} because $\dim \mcal A= \dim \mcal A_l-2$. 

Let $l_0$ be such that $\rho(Z_{l_0})$ is maximal. Then, again by Proposition~\ref{pro:2.1}, both $Z_{l_0}\dashto X_i$ and $Z_{l_0}\dashto X_{j}$ are compositions of finitely many elementary contractions and by definition of $l_0$, these are birational contractions. Since $\rho(X_i)$ and $\rho(X_j)$ are both $\geq \rho(T_{i,j})+1$, there are indices $1\leq i',j'\leq n$ and elementary contractions $Z_{i'}=\widetilde{X_i}\to X_{i}$ and $Z_{j'}=\widetilde{X_j}\to X_{j}$ with $Z_{l_0},\widetilde{X_i},\widetilde{X_j}$ isomorphic in codimension $1$, that fit in the diagram: 
\begin{equation*}
\xymatrixrowsep{0.15in}
\xymatrixcolsep{0.2in}
\xymatrix{\widetilde{X_i} \ar@{-->}[rr] \ar[d] & &\widetilde{X_{j}}\ar[d]\\
X_{i}\ar[d] & &X_{j} \ar[d]\\
S_{i}\ar[dr] & & S_{j}\ar[dl] \\
 & T_{i,j} & }\end{equation*} 
This shows that every edge of $\mcal N_{\mathfrak R}$ defines a Sarkisov link dominated by $Z$. 

Conversely, let $L_{i,j}\colon X_i/S_i\dashto X_j/S_j$ be an elementary Sarkisov link dominated by $Z$. Let $v_i$ be the vertex of $\mcal N_\mathfrak R$ associated to $X_i/S_i$ and $v_j$ that associated to $X_j/S_j$, and as above, denote by $\overline{\mcal A_{i_1}}, \overline{\mcal A_{j_1}}$ the dual facets of $\mcal B^+(\mcal C_\mathfrak R)$. By Definition~\ref{dfn:3.2}, $\overline{\mcal A_{i_1}}\cap \overline{\mcal A_{j_1}}\neq \emptyset$. But then, by definition of the decomposition \eqref{eq:3.2}, there is an index $k\in \mcal B^+(I)$ with $\overline{\mcal A_{k}}= \overline{\mcal A_{i_1}}\cap \overline{\mcal A_{j_1}}$, and $\dim \mcal A_k= \rho(Z)-2$. The vertices $\{v_i,v_j\}$ span a $1$-simplex dual to $\overline{\mcal A_k}$.
\end{proof}

\begin{cor}\label{cor:3.1}
Let $[v_i, v_j]$ be a $1$-simplex of $\mcal N_\mathfrak R$, $\overline{\mcal A}\in \Ob \mcal B^+(\mcal C_\mathfrak R)$ its dual face and let $\vphi\colon Z\lto T_{i,j}$ be the ample model associated to $\mcal A$.
The elementary Sarkisov link associated to $[v_i,v_j]$ is of:
\begin{enumerate}
\item[] type I (resp. type III) if $\vphi_i\simeq \vphi\not \simeq \vphi_j$(resp. $\vphi_i\not \simeq \vphi \simeq \vphi_j$), 
\item[] type II if $\vphi_i\simeq \vphi_j$, and type IV if $\vphi_i\not \simeq  \vphi \not \simeq \vphi_j$.
\end{enumerate}
\end{cor}
\begin{proof} This is a reformulation of the definition of the types of links.
\end{proof}

The next Proposition is an extension of Lemma~\ref{lem:3.3} to affine spaces of arbitrary dimension. Explicitly, if
$H$ is in general position with respect to $\mcal C_\mathfrak R$, there is a subring $\mathfrak R_H$ whose support $\mcal C_{\mathfrak R,H}$ is the cone over $\mcal C_\mathfrak R\cap H$. There is a natural geography of models associated to the convex geometry of $\mcal C_{\mathfrak R,H}$, which behaves like the geography of models of a divisorial ring on some $Z'$ of Picard rank $\rho(Z')= \dim H+1$. 

\begin{pro}\label{pro:3.3} Let $H$ be an affine space of dimension $d\leq \rho(Z)-1$ in general position with respect to $\mcal C_\mathfrak R$. There is a finitely generated ring $\mathfrak R_H$ whose support $\mcal C_{\mathfrak R, H}$ has dimension $d+1$ and a decomposition  
\begin{equation}\label{eq:3.4}
\mcal C_{\mathfrak R,H}=\coprod \mcal A'_i
\end{equation}
into cones $\mcal A'_i$ with base $\mcal A_i\cap H$, where $\mcal A_i$ are the cones in  \eqref{eq:3.2}; the cones $\mcal A'_i$ in \eqref{eq:3.4} satisfy (1--4) in Theorem~\ref{thm:2.1}. 

For every $\mcal A_i$ in \eqref{eq:3.2}, when $\{0\}\neq\mcal A'_i$, the codimension of $\mcal A'_i$ in $\mcal C_{\mathfrak R,H}$ is that of $\mcal A_i$ in $\mcal C_\mathfrak R$. 
If $\vphi_i\colon Z\lto Z_i$ is the ample model associated to $\mcal A'_i$ with $\dim \mcal A'_i= d+1$, $Z_i$ is $\Q$-factorial. 

There are polyhedral complexes $\mcal C_{\mathfrak R,H}$, $\mcal B^+(\mcal C_{\mathfrak R, H})$ naturally associated to \eqref{eq:3.4}; $ \mcal N_{\mathfrak R, H} =\Ner (\mcal B^+(\mcal C_{\mathfrak R, H}))^{(d)}$ is a subcomplex of $\mcal N_\mathfrak R$.  
 \end{pro}
\begin{proof} Since $H$ is in general position with respect to $\mcal C_\mathfrak R$, $H\cap \mcal C_\mathfrak R$ is a rational polytope of dimension $\rho(Z)-d-1$ whose vertices are of the form $\varepsilon_1(K_Z+\Delta_1), \dots, \varepsilon_m(K_Z+\Delta_m)\in \mcal C_\mathfrak R$ for $\varepsilon_i\in \Q$, $(Z,\Delta_i)$ klt and $\Delta_i$ big $\Q$-divisors. Define $\mathfrak R_H= R(Z; K_Z+\Delta_1, \dots, K_Z+\Delta_m)$; then $\mathfrak R_H$ is finitely generated and its support is the cone over $H\cap \mcal C_\mathfrak R$.
The properties of the cones in the decomposition of $\mathfrak R_H$ then follow from the fact that $H$ is in general position. Also, if $\mcal A_i\subset \overline{\mcal A_j}$ and $\mcal A_i\cap H\neq \emptyset$, then $\mcal A_j\cap H\neq \emptyset$ and $\dim \mcal A_j-\dim \mcal A_i= \dim \mcal A'_j- \dim \mcal A'_i$, so that $\mcal N_{\mathfrak R, H}$ is a subcomplex of $\mcal N_\mathfrak R$.  

 \end{proof}
\begin{dfn} Let $\mcal C$ be a polyhedral complex, and $\overline{\mcal A}\in \Ob(\mcal C)$. The \emph{residual complex} $\resi_\mcal A(\mcal C)$ is the minimal subcomplex of $\mcal C$ containing all objects $c\in \Ob \mcal C$ such that there is a map $\{f\colon \mcal A\lto c\} \in \Mor \mcal C$. \end{dfn}

\begin{lem} \label{lem:3.4}Let $\mcal A= \mcal A_{i_0}$ be a cone in \eqref{eq:3.3}, where $i_0\in \mcal B^+(I)$. 
\begin{enumerate}
\item[1.] Let $\overline{\mcal A_1}, \overline{\mcal A_2}$ be facets of $\resi_\mcal A \mcal C_\mathfrak R$ (and hence of $\mcal C_\mathfrak R$). Then, $Z_1\dashto Z_2$ is the composition of a finite number of maps $Z_j\dashto Z_{j'}$ that are either isomorphisms in codimension $1$, morphisms that contract a single divisor, or inverses of morphisms contracting a single divisor and for all $j,j'$, the diagram 
\[
\xymatrixrowsep{0.2in}
\xymatrixcolsep{0.3in}
\xymatrix{ 
Z_j \ar@{-->}[rr] \ar[dr] & \quad & Z_{j'}
  \ar[dl]\\
\quad & Z_{i_0} & \quad
}\]
commutes.  
\item[2.] If $H$ is an affine subspace in general position with respect to $\mcal C_\mathfrak R$ and $\mcal A$, then $\resi_\mcal A\mcal C_{\mathfrak R}\subseteq \mcal C_{\mathfrak R, H}$.
\item[3.] If $\mcal A_1, \dots , \mcal A_n$ are the cones in \eqref{eq:3.2} such that $H$ is in general position with respect to $\mcal A_1, \dots , \mcal A_n$, then $ \mcal C_{\mathfrak R, H}= \cup \resi_{\mcal A_i}(\mcal C_\mathfrak R)$. 
\end{enumerate}
\end{lem}
\begin{proof}
For (1), recall that by (3) in Theorem~\ref{thm:2.1}, the diagram  
\[
 \xymatrixrowsep{0.15in}
\xymatrixcolsep{0.25in}
\xymatrix{\quad &Z\ar[dl]_{\vphi_1}\ar[dr]^{\vphi_{2}}& \quad \\ 
Z_1 \ar@{-->}[rr] \ar[dr]_{\vphi_1, i_0} & \quad & Z_{2}
  \ar[dl]^{\vphi_{2, i_0}}\\
\quad & Z_{i_0} & \quad}
\]
commutes; the result follows immediately from the decomposition of the morphisms $\vphi_1, \vphi_2$ into elementary contractions of Proposition~\ref{pro:2.1}(4).
(2--3) follow from the fact that $H$ is in general position and from the definition of the decomposition \eqref{eq:3.2}. 
\end{proof}

\begin{dfnlm} \label{dfnlm:3.1}Let $\mcal A=\mcal A_{i_0}$ be a cone in \eqref{eq:3.3} of dimension $d$, where $i_0\in \mcal B^+(I)$. There is an $(\rho(Z)-d-1)$-dimensional affine subspace $H_\mcal A$ in general position with respect  to $\mcal C_\mathfrak R$ such that the cones $\mcal A'_i$ in the decomposition  \eqref{eq:3.4} of $\mcal C_{\mathfrak R,{H_\mcal A}}$ are in $1$-to-$1$ correspondence with the objects of $\resi_\mcal A \mcal C_\mathfrak R$. 
The ring $\resi_\mcal A \mathfrak R=  \mathfrak R_{H_\mcal A}$ is \emph{residual to $\mcal A$}. 
\end{dfnlm}
\begin{proof} 
Let $H$ be a $(\rho(Z)-d)$-dimensional affine subspace in general position with respect to $\mcal C_\mathfrak R$ such that $\mcal A\cap H$ is $0$-dimensional and $\mcal C_\mathfrak R\cap H$ is $(\rho(Z)-d)$-dimensional. The intersection $\mcal A\cap H$ is convex, so that $\mcal A\cap H= \{D\}$ is a single point. 

Let $\mcal A_1,\dots, \mcal A_n$ be the cones distinct from $\mcal A$ in \eqref{eq:3.3} that are of dimension $d$ and that intersect $H$. Denote by $D_1, \dots , D_n$ the points $\mcal A_1\cap H, \dots , \mcal A_n\cap H$: $D$ and $D_1, \dots, D_n$ are vertices of the polytope $\mcal C_\mathfrak R\cap H$.  Therefore, we may choose $H_\mcal A\subseteq H$ a suitably small translation of the supporting hyperplane of $\mcal C_\mathfrak R\cap H$ at $D$ that is in general position with respect to $\mcal C_\mathfrak R$, and such that $D$ belongs to one of the half-space of $H$ defined by $H_\mcal A$, the vertices $D_1, \dots, D_n$ to the other half-space, and the only cones $\mcal A_i'$ that have nonempty intersection with $H_\mcal A$ contain $D$ (i.e.~are of the form $\mcal A'_i= \mcal A_i\cap H$, with $\mcal A\subset \overline{\mcal A_i}$). 

Let $ \mcal C_{\mathfrak R,{H_\mcal A}}= \coprod \mcal A"_i$
be the decomposition induced by \eqref{eq:3.2} on the support of $\mathfrak R_{H_\mcal A}$. Then, by construction and by (2) in Lemma~\ref{lem:3.4}, there is a $1$-to-$1$ correspondence between cones $\{0\}\neq \mcal A"_i$ and objects $\overline{\mcal A}\neq \overline{\mcal A_i}$ of $\resi_\mcal A \mcal C_\mathfrak R$; this is what we wanted. 
\end{proof}
The next Lemma is a reformulation of a key idea in \cite{HM09}, and implies Theorem~\ref{thm:0.2}. It shows that, similarly to Definition-Lemma~\ref{dfnlm:2.1}, decompositions of birational maps between Mori fibre spaces correspond to edge paths on $\mcal N_\mathfrak R$--or equivalently, as in Remark~\ref{rem:2.6}, to suitable paths on $\partial^+\mcal C_\mathfrak R$.
\begin{lem} \label{lem:4.1}
Assume the setup of Notation~\ref{nt}. 
Let $\Phi\colon X/S\dashto Y/T$ be a map in the collection ($\Phi=\Phi_{j,j'}$ for some $j,j'$).
If $v_{X/S}$ and $v_{Y/T}$ are the vertices of $\mcal N_\mathfrak R$ associated to $X/S$ and $Y/T$, there is an edge path $(v_0, \dots, v_n)$ on $\mcal N_\mathfrak R$ with $v_0= v_{X/S}$ and $v_n=v_{Y/T}$ on $\mcal N_\mathfrak R$. Further,
$\Phi\simeq L_{n-1, n}\circ \dots \circ L_{0,1},$
where $L_{j,j+1}$ is the elementary Sarkisov link associated to the $1$-simplex $\{v_j, v_{j+1}\}$.

There is a $1$-to-$1$ correspondence between edge paths on $\mcal N_\mathfrak R$ and the decompositions of birational maps $X_j/S_j\dashto X_{j'}/S_{j'}$ into elementary Sarkisov links dominated by $Z$. 
\end{lem}
\begin{proof} Let $\overline{\mcal A_S}$ and $\overline{\mcal A_T}$ be the facets of $\mcal B^+(\mcal C_\mathfrak R)$ dual to $v_{X/S}$ and $v_{Y/T}$, and $\overline{\mcal A_X}, \overline{\mcal A_Y}$ the associated facets of $\mcal C_\mathfrak R$. Denote by $f\colon Z\lto X$ and $g\colon Z\lto Y$ the ample models associated to $\mcal A_X, \mcal A_Y$. 

I first show that if $\Phi$ is dominated by $Z$, there is a well defined edge path on $\mcal N_\mathfrak R$ from $v_{X/S}$ to $v_{Y/T}$ that decomposes $\Phi$ into elementary links. It is enough to prove that there is a $2$-dimensional affine subspace $H$ of $V$ in general position with respect to $\mathfrak R$  such that $v_{X/S}$ and $v_{Y/T}$ belong to the same connected component of $\mcal N_{\mathfrak R,H}$.  Dually, this is equivalent to showing that $\mcal A_S\cap H$ and $\mcal A_T\cap H$ belong to the same connected component of $ \partial^+ \mcal P_\mathfrak R\cap H$. 

By (3) in Proposition~\ref{pro:3.1}, there is a klt pair $(Z, \Delta)$, with $K_Z+\Delta\in V$ not pseudoeffective, and ample $\Q$-divisors $A_X$ and $A_Y$ such that $f$ and $g$ are results of log-MMPs for $(Z, \Delta)$. These results of log-MMPs are MMPs with scaling by $A_X$ and $A_Y$, i.e.~$K_Z+\Delta+A_X\in \mcal A_S$ and $K_Z+\Delta+A_Y\in \mcal A_T$; this forces $K_Z+\Delta+(1+\varepsilon)A_X\in \mcal A_X$ and $K_Z+\Delta+(1+\varepsilon)A_Y\in \mcal A_T$ for $0<\varepsilon <\!<\!1$. 
By Lemma~\ref{lem:3.2}, up to small perturbation of $A_X, A_Y$, and $\Delta$, we may further assume that $H= (K_Z+\Delta) + \R_+A_X+ \R_+ A_Y \subseteq V$ is in general position with respect to $\mcal C_\mathfrak R$. Denote by $\mcal P_{\mathfrak R,H}= \mcal C_\mathfrak R\cap H$ and by $\mcal Q_i= \mcal A_i\cap H$. By construction, $\mcal Q_S$ and $\mcal Q_T$ and hence also $\partial^+ \mcal P_{\mathfrak R,H}$ are $1$-dimensional. 
Since $K_Z+\Delta\not \in \mcal P_{\mathfrak R,H}$, $$\mcal T= \{K_Z+\Delta +\lambda (tA_X+ (1-t)A_Y)\mid \lambda>0, t\in [0,1]\} \cap \partial \mcal P_{\mathfrak R,H}$$ is $1$-dimensional, and intersects $\mcal Q_S$ and $\mcal Q_T$ in a $1$-dimensional locus. A divisor $D\in \mcal T$ is not big, as this would imply that there is $t\leq 1$ such that $K_Z+\Delta+ tA_X$ or $K_Z+\Delta+tA_Y$ is big. Therefore, $\mcal T \subseteq \partial^+ \mcal P_{\mathfrak R,H}$ and $\mcal Q_S$ and $\mcal Q_T$ belong to the same component of $\partial^+\mcal P_{\mathfrak R,H}$.     

Since $v_{X/S}$ and $v_{Y/T}$ belong to the same path-connected component of $\mcal N_\mathfrak R$, there is an edge path $(v_0, \dots, v_n)$ on $\mcal N_{\mathfrak R}$ with $v_0=v_{X/S}$ and $v_n=v_{Y/T}$. Since $\Phi\simeq f_{n}\circ f_{0}^{-1}$ and since each $L_{j,j-1}\simeq f_{j}\circ f_{j-1}^{-1}$, $\Phi\simeq L_{n, n-1}\circ \dots \circ L_{0,1}$ is a decomposition of $\Phi$ into elementary Sarkisov links. 

Conversely, an edge path of $\mcal N_\mathfrak R$ defines a birational map $\Phi$ which is the composition of the elementary Sarkisov links associated to its edges, and $\Phi$ is dominated by $Z$ by construction. 
 \end{proof}

\section{Relations in the Sarkisov Program}
\label{relations}
Let $\Phi_{j,j'}\colon X_j/S_j \dashto X_{j'}/S_{j'}$ be a birational map between Mori fibre spaces;  by Theorem~\ref{thm:0.2}, $\Phi_{j,j'}$ is the composition of a finite number of elementary Sarkisov links. However, in general, this decomposition needs not be unique. In this section, I study \emph{relations in the Sarkisov program}, i.e.~distinct factorisations of this form of a given $\Phi_{j,j'}$. More precisely, I show that a factorisation of $\Phi_{j,j'}$ into elementary Sarkisov links corresponds to an \emph{edge path} on a suitable simplicial complex $\mcal N_\mathfrak R$. As a consequence, \emph{relations} are determined by the fundamental group of $\partial^+\mcal C_\mathfrak R$, the locus of $\mcal C_\mathfrak R$ that consists of non-zero divisors that are not big, where $\mathfrak R$ is a suitable divisorial ring.

\subsection{Sarkisov Program and edge paths on $\mcal N_{\mathfrak R}$} \label{subsec:4.1}
In this Section, I define elementary relations in the Sarkisov program and show that they correspond to generators of the edge path group of $\mcal N_\mathfrak R$.

\begin{dfn} \label{dfn:4.2}
A \emph{non-trivial relation in the Sarkisov Program} is a composition of $r>2$ Sarkisov links
\begin{equation}\label{eq:4.1} \id \simeq L_{r-1,r}\circ \dots \circ L_{0,1}\end{equation}
which defines an automorphism of $X_0\simeq X_r$ that commutes with $X_0\to S_0$.

The relation \eqref{eq:4.1} is \emph{elementary} if, in addition, no proper subchain of links in \eqref{eq:4.1} forms a relation, i.e.~$L_{j, j-1}\circ \dots \circ L_{i+1, i}\not \simeq \id$ for any $1< i<j\leq r$, and if for all $i$, $L_{i+1,i}\circ L_{i,i-1}$ is not an elementary Sarkisov link. 

The relation is \emph{dominated by} a normal projective variety $Z$ if $Z$ dominates all the elementary Sarkisov links $L_{j,j-1}$. \end{dfn}

Consider a relation in the Sarkisov Program, and let $Z$ and $\mathfrak R$ be associated to the collection $\{X_1/S_1, \dots, X_{r-1}/S_{r-1}; L_{j+1,j}\}$ as in Proposition~\ref{pro:3.1}. 
By Lemma~\ref{lem:4.1}, we may assume that $\mcal N_\mathfrak R$ is connected (or equivalently that $\partial^+\mcal C_\mathfrak R$ is connected and purely $(\rho(Z)-1)$-dimensional) to study birational maps dominated by $Z$. This is not restrictive: given a map $\Phi\colon X/S\dashto Y/T$ dominated by $Z$, $\mathfrak R$ may be replaced by a subring to get rid of components of $\partial^+\mcal C_\mathfrak R$ that do not contain $v_{X/S}$ and $v_{Y/T}$. This is achieved as in the proof of (4) in Proposition~\ref{pro:3.1}.
The next result is a simple consequence of Lemma~\ref{lem:4.1}.
\begin{cor}\label{cor:4.1} There is a $1$-to-$1$ correspondence between elementary relations in the Sarkisov Program dominated by $Z$ and equivalence classes of edge paths $(v_0, \dots , v_n)$ on $\mcal N_{\mathfrak R}$ with $v_0= v_n$ and $v_i\neq v_j$ for any $0<i<j\leq n$. Equivalently, there is a $1$-to-$1$ correspondence between elementary relations in the Sarkisov Program dominated by $Z$ and generators of the fundamental group $\pi_1(\partial^+\mcal P_\mathfrak R)$. When $\rho(Z)\geq 4$, $\pi_1(\partial^+ \mcal P_\mathfrak R)\simeq \pi_1(\partial^+ \mcal C_\mathfrak R)$. 
\end{cor}
\begin{proof}
Lemma~\ref{lem:4.1} associates to any non-trivial relation dominated by $Z$ an edge loop $(v_0, \dots , v_n)$ on $\mcal N_\mathfrak R$ based at $v_0$. 
When the relation is elementary, $(v_0, \dots, v_n)$ is a simple edge loop, because $v_i=v_j$ for $0\leq i<j<N$ would contradict the definition. Also, since $L_{i+1, i}\circ L_{i,i-1}$ is not a Sarkisov link, $\{v_i, v_{i+1}, v_i\}$ never spans a simplex, and $(v_0, \dots, v_n)$ is the ``minimal" edge loop in its equivalence class.
This shows that an elementary relation defines the equivalence class of a generator of $E(\mcal N_\mathfrak R)$.The converse is clear by constuction of $\mcal N_\mathfrak R$. 
Since $\mcal N_\mathfrak R$ is connected, $E(\mcal N_\mathfrak R, v)$ is independent of the choice of $v=v_0$.

For the last assertion, recall that $E(\mcal N_\mathfrak R)\simeq \pi_1(|\mcal N_\mathfrak R|)$, and $\pi_1(|\mcal N_\mathfrak R|)\simeq \pi_1(|\mcal B^+(\mcal Q_\mathfrak R)|)$ by Lemma~\ref{lem:2.0}. But, as noted in Remark~\ref{rem:3.6}, $\partial^+\mcal P_\mathfrak R$ is the underlying space of $\mcal B^+(\mcal Q_\mathfrak R)$ . Similarly, $E(\Ner \mcal B^+(\mcal C_\mathfrak R))\simeq \pi_1(\partial^+\mcal C_\mathfrak R)$, and since the edge path group of a simplicial complex only depends on its $2$-skeleton, $E(\mcal N_\mathfrak R)= E(\Ner \mcal B^+(\mcal C_\mathfrak R))$ when $\rho(Z)\geq 4$.
\end{proof}

Standard results in Algebraic Topology give the following description of $E(\mcal N_\mathcal R)$. The $1$-skeleton of $\mcal N_\mathfrak R$ is a connected combinatorial graph. Let  $T_\mathfrak R$ be a \emph{spanning tree}, i.e.~an edge path $(v_0, \dots, v_n)$ containing all vertices of $\mcal N_\mathfrak R$ but no cycle, and let $A$ be the set of edges of $\mcal N_{\mathfrak R}$ that are not in $T_\mathfrak R$. Then, $E(\mcal N_\mathfrak R)$ is the set of equivalence classes of the group freely generated by the \emph{fundamental cycles} $\{\gamma_{i,j}\mid \{v_i, v_j\}\in A\}$, where $\gamma_{i,j}$ is the class of the cycle obtained by adding $\{v_i, v_j\}\in A$ to $T_{\mathfrak R}$.

\begin{cor}\label{cor:4.2} If $\dim \mcal C_\mathfrak R= 3$, $E(\mcal N_\mathfrak R)$ is either trivial or generated by a single element $\gamma$. 
\end{cor}

\begin{proof} When $\dim \mcal C_\mathfrak R= 3$, $\partial \mcal P_\mathfrak R$ is homeomorphic to $S^1$, and $E(\mcal N_\mathfrak R)$ is trivial when $\partial^+\mcal P_{\mathfrak R}\subsetneq \partial\mcal P_\mathfrak R$ and generated by an element $\gamma$ otherwise. The path $\gamma$ is naturally dual to the complex $\mcal B^+(\mcal Q_\mathfrak R)$. 
\end{proof}
\begin{exa}\label{exa:4.1'} Recall the setting of Example~\ref{exa:4.1}.
The group $E(\mcal N_\mathfrak R)$ is generated by the unique fundamental cycle through all vertices, and there is a unique elementary relation of the Sarkisov program dominated by $Z$, which is the relation \eqref{eq:1.2} mentionned in the introduction.\end{exa}
\begin{conv}
So far, I always assumed that $\dim \mcal C_\mathfrak R\geq 2$. When $\mcal C_\mathfrak R$ is $2$-dimensional, the situation is that studied in the $2$-ray game, and there is at most one Sarkisov link: $\mcal N_\mathfrak R$ consists of at most one vertex.  Set $E(\mcal N_\mathfrak R)$ to be the trivial group when $\dim \mcal C_\mathfrak R=2$.
\end{conv}

\subsection{Relations on $3$-dimensional rings}
In this section, I give a geometric description of $E(\mcal N_\mathfrak R)$ when $\rho(Z)=3$.  
There is a fairly explicit description of $2$-ray configurations on varieties $Z$ with $\rho(Z)=2$ and rational polyhedral pseudoeffective cone. This leads to the definition of $4$ types of elementary Sarkisov links. When $\rho(Z)=3$, there is no such description: any attempt at a \emph{classification} elementary relations in the Sarkisov Program is essentially qualitative.

 \begin{dfn}\label{dfn:4.3} 
 Let $\gamma= (v_0, \dots, v_n)$ be an edge path on $\mcal N_\mathfrak R$, and denote by $\overline{\mcal A_i}= v_i^*$, and by $Z\lto S_i$ the ample model associated to $\mcal A_i$. 
Assume that $\rho(S_0)=\min \{\rho(S_j)\}$, and set $S= S_0$ and $\rho=\rho(S_0)$.\\
The path $\gamma$ is \emph{of type A} if for some $n_1<n_2< n$ or $n_1=n_2=n$:
\begin{itemize} \item[] $\{v_{j-1}, v_j\}$ is of type II or IV unless $j=n_1$ or $n_2$, 
\item[] $\{v_{n_1-1},v_{n_1}\}$ is of type I, and $\{v_{n_2-1}, v_{n_2}\}$ is of type III.
\end{itemize}
Then, $\rho(S_j)= \rho+1$ for $n_1\leq j\leq n_2-1$, and $\rho(S_j)= \rho$ otherwise.\\
The path $\gamma$ is \emph{of type B} if there are $n_1<n_2< n$ with:
\begin{itemize} \item[] $\{v_{j-1}, v_j\}$ is of type II or IV for $j<n_1$ and $j>n_2$, \item[] $\{v_{n_1-1},v_{n_1}\}$ is of type I and $\{v_{n_2-1}, v_{n_2}\}$ of type III, \item[] $(v_{n_1}, \dots, v_{n_2-1})$ is the product of finitely many paths of type A.
\end{itemize} 
Then, $\rho(S_j)= \rho+1$ or $\rho+2$ for $n_1\leq j\leq n_2-1$, and $\rho(S_j)= \rho$ otherwise. 

\end{dfn}
\begin{cor}\label{cor:4.3} Assume that $\dim \mcal C_\mathfrak R=3$. A non-trivial edge loop $(v_0, \dots, v_n)\in E(\mcal N_\mathfrak R)$ is the composition of finitely many paths of type A or B.
An elementary relation $L_{n,n-1}\circ \dots \circ L_{1,0}$ of the Sarkisov Program dominated by $Z$ is the composition of finitely many chains of elementary links of type A or B.
\end{cor}
\begin{proof}Let $X_j/S_j$ be the Mfs associated to each vertex $v_j$ and assume that $\rho(S_0)=\rho= \min \{\rho(S_j)\}$. 
As noted above, $E(N_\mathfrak R)$ is non trivial precisely when $\partial^+ \mcal P_\mathfrak R= \partial P_\mathfrak R$ is homeomorphic to $S^1$. Consider $\partial \mcal P_\mathfrak R$ equipped with the triangulation induced by the decomposition \eqref{eq:3.2}; then $\partial \mcal P_\mathfrak R$ is a simplicial complex dual to $\mcal N_\mathfrak R$. 
There are two cases. If $S_0$ is not a point, $\rho\geq1$ and the generator of $E(\mcal N_\mathfrak R)$ (i.e.~the dual of $\partial \mcal P_\mathfrak R$) is the composition of a finite number of edge paths of the form $L_1\cup L_2 \cup L_3$,
where $L_1,L_2$ and $L_3$ are connected, and for each $v_i\in L_1\cup L_3$ (resp. $v_i\in L_2$), $\rho(S_i)=1$ (resp. $\rho(S_i)=2$); $L_1\cup L_2\cup L_3$ is of type A.

If $\rho=0$, the generator of $E(\mcal N_\mathfrak R)$ is the composition of finitely many edge paths of the form $L_1\cup L_2 \cup L_3$,
where each $L_i$ is connected, $L_2$ is a path of type A, and $S_i$ is a point (resp. $\rho(S_i)=1$ or $2$) for all $v_i\in L_1\cup L_3$ (resp. $v_i\in L_2$); $L_1\cup L_2\cup L_3$ is of type B. 
\end{proof}

\begin{cor}\label{cor:4.4} 
Assume that $\dim \mcal C_\mathfrak R= \rho\geq 3$ and let $\mcal A$ be a cone in \eqref{eq:3.2} of dimension $\rho-3$ that contains no big divisor. 
Then, $E(\resi_\mcal A \mcal N_\mathfrak R)$ is trivial or is the free group generated by an edge-loop $(v_0, \dots , v_n)$ which is the composition of finitely many chains of type A or B.
\end{cor}
\begin{proof}
Recall from Definition-Lemma~\ref{dfnlm:3.1} that $\resi_\mcal A \mathfrak R$ is a finitely generated ring with $3$-dimensional support. By Corollary~\ref{cor:4.2}, $E(\resi_\mcal A\mcal N_\mathfrak R)$ is trivial or generated by a single edge loop $(v_0, \dots, v_n)$. Denote by $Z\lto S$ the ample model associated to the cone $\mcal A$. By Lemma~\ref{lem:3.4}, the proof of Corollary~\ref{cor:4.3} applies to this situation after replacing ``point" by $S$ because all relevant contractions occur over $S$.
\end{proof}
The goal of Example~\ref{exa:4.2} is to show that both types of relations in the Sarkisov Program are realized among the end products of the MMP on a terminal $\Q$-factorial Fano $3$-fold $Z$ with $\rho(Z)=3$. Let $\{X_i/S_i\}$ be the Mori fibre spaces that are results of the MMP on $Z$. The $3$-fold $Z$ admits no small contraction and is a common resolution of all $X_i$. If $D_1, D_2, D_3$ is a basis of $\Effb(Z)$ and $A\sim_{\Q}{-}K_Z$, the ring $\mathfrak R(K_Z+A+D_1, K_Z+A+D_2, K_Z+A+D_3)$ is finitely generated and has support $\Effb_\R(Z)$. Then, $Z$, $\mathfrak R$ are associated to the collection of $X_i/S_i$ and Sarkisov links dominated by $Z$.  
\begin{exa}\label{exa:4.2}
The methods of \cite{MM86} can be used to classify the polytopes $\mcal P_\mathfrak R$ and complexes $\mcal Q_\mathfrak R$. The numbers between parentheses are those in the classification of \cite{MM82}, and $n$ is the number of chambers of maximal dimension in the coarsest decomposition \eqref{eq:3.2}. In the figures, the faces $\mcal Q_i$ of dimension $1$ and $2$ are labelled by the associated ample models except when this model is $Z\lto \{P\}$. 
\begin{enumerate}
\item[Case 1:] $Z$ admits no E$1$-contraction to a Fano $3$-fold (Figure 3). 
Either $Z$ admits no E$1$ contraction and $n=1$ (1,27), or $Z$ admits an E$1$-contraction to a weak Fano $3$-fold and $n=3$ (2,31). The relation is of type A. 
\begin{figure}[h]\begin{center}
\vspace{-10pt}
\begin{pspicture}[psscale=.5](-1,-0.5)(9,3.2)
\psset{unit=0.75cm}
\psdots[dotscale=1](0,0)(2,3)(4,0)
\pspolygon(0,0)(2,3)(4,0)
\rput[t](2,-0.2){$\PS^1\times \PS^1$}
\rput[br](0.8,1.5){$\PS^1\times \PS^1$}
\rput[bl](3.1,1.5){$\PS^1\times \PS^1$}
\rput[b](2,3.15){$\PS^1$}
\rput[tr](0, -0.2){$\PS^1$}
\rput[tl](4,-0.2){$\PS^1$}
\rput(2,1){$Z$}
\pspolygon(7,0)(9,3)(11,0)
\psline[linewidth=0.4pt](7,0)(9,1)(9,3)
\psline[linewidth=0.4pt](9,1)(11,0)
\rput[t](9,-0.2){$\PS^1\times \PS^1$}
\rput[br](7.8,1.5){$\PS^1$}
\rput[bl](10.1,1.5){$\PS^1$}
\rput[b](9,3.15){$P$}
\rput[tr](7, -0.2){$\PS^1$}
\rput[tl](11,-0.2){$\PS^1$}
\rput(9,0.5){$Z$}
\rput(8,1){$X_1$}
\rput(10,1){$X_2$}
\end{pspicture}
\end{center}
\vspace{-10pt}
\caption{}
\vspace{-10pt}
\end{figure}

\item[Case 2:] $Z$ is not primitive and admits a structure of Mori fibre space (Figure 4). There are $3$ cases:  $n=2$ (4,28), $n=3$ (3,8,17,24) and $n=4$ (6,10,25,30). Here, $\rho(X_i)=2$ and $\rho(V)=1$. The relation is of type A in the first two cases, and of type $B$ in the last one. 

\begin{figure}[h]\vspace{-10pt}
\begin{center}
\begin{pspicture}[psscale=.5](-1,-3.3)(10,2.5)\psset{unit=0.75cm}
\psdots[dotscale=1](0,0)(1,2)(4,2)(5,0)
\pspolygon(0,0)(1,2)(4,2)(5,0)
\psline[linewidth=0.4pt](1,2)(5,0)
\rput[br](0.5,1){$\PS^1$}
\rput[bl](4.7,1){$\F_1$}
\rput[b](2.5,2.2){$\PS^1\times \PS^1$}
\rput[t](2.5,-0.2){$\PS^2$}
\rput(1.5,0.75){$X_1$}
\rput(3.5,1.25){$Z$}
\psdots[dotscale=1](7,0)(8,2)(11,2)(12,0)(9.5,0)
\pspolygon(7,0)(8,2)(11,2)(12,0)
\psline[linewidth=0.4pt](8,2)(9.5,0)(11,2)
\rput[t](9.5,-0.2){$\PS^2$}
\rput(9.5,1){$Z$}
\rput(8.25,0.8){$X_1$}
\rput(10.75,0.8){$X_2$}
\rput[br](7.5,1){$\PS^1/\PS^2$}
\rput[bl](11.7,1){$\PS^1$}
\rput[b](9.5,2.1){$\PS^1\times \PS^1/\F_1$}
\pspolygon(3.5,-3)(4.5,-1)(7.5,-1)(8.5,-3)
\psline[linewidth=0.4pt](4.5,-1)(6,-2.25)(7.5,-1)
\psline[linewidth=0.4pt](3.5,-3)(6,-2.25)(8.5,-3)
\rput[br](4,-2){$\PS^1/\PS^2$}
\rput[bl](8.2,-2){$\PS^1$}
\rput[b](6,-0.9){$\PS^1\times \PS^1/\F_1$}
\rput[t](6,-3.2){$P$}
\rput(6, -2.7){$V$}
\rput(6, -1.5){$Z$}
\rput(5,-2){$X_1$}
\rput(7,-2){$X_2$}
\end{pspicture}
\end{center}\vspace{-10pt}
\caption{}
\vspace{-10pt}
\end{figure}
\item[Case 3:] $Z$ is not primitive and has no Mori fibre space structure (Figure 5). There are $2$ cases with $n=5$ that correspond to whether the interior chambers yield ample models with $\rho=2$ only (7,13) or not (11,12,15,16,20,26,22), one with $n=6$, and one with $n=8$. Again, $\rho(X_i)= 2$ and $\rho(V_j)=1$.  The relation is always of type A. 
\begin{figure}[h]\begin{center}
\vspace{-10pt}
\begin{pspicture}[psscale=.5](-1,-0.5)(9.5, 3)\psset{unit=0.75cm}
\psdots[dotscale=1](0,0)(1,1.5)(2,3)(3,1.5)(2,0)(4,0)
\pspolygon(0,0)(2,3)(4,0)
\pspolygon[linewidth=0.4pt](2,0)(3,1.5)(1,1.5)
\rput[t](2,-0.2){$\PS^1/\PS^2$}
\rput[br](0.8,1.5){$\PS^2$}
\rput[bl](3.3,1.5){$\PS^2$}
\rput(2,1){$Z$}
\rput(2,2){$X_1$}
\rput(1,0.5){$X_2$}
\rput(3,0.5){$X_3$}
\psdots[dotscale=1](7,0)(8,1.5)(9,3)(10,1.5)(11,0)
\psline(7,0)(9,3)(11,0)(7,0)
\psline[linewidth=0.4pt](7,0)(10,1.5)(8,1.5)(11,0)
\rput[br](7.8,1.5){$\PS^2/\PS^1$}
\rput[bl](10.4,1.5){$\PS^2$}
\rput(9,0.5){$V$}
\rput(9,1.3){$Z$}
\rput(9,2){$X_1$}
\rput(8.2,1){$X_2$}
\rput(9.8,1){$X_3$}
\end{pspicture}
\end{center}
\begin{center}
\begin{pspicture}[psscale=.5](-1,-0.5)(9.5, 3)\psset{unit=0.75cm}
\psdots[dotscale=1](0,0)(2,3)(4,0)
\pspolygon(0,0)(2,3)(4,0)
\pspolygon[linewidth=0.4pt](2,0)(1.7,1.5)(2.3,1.5)
\pspolygon[linewidth=0.4pt](2,3)(1.7, 1.5)(2.3,1.5)
\psline[linewidth=0.4pt](0,0)(1.7,1.5)
\psline[linewidth=0.4pt](2.3,1.5)(4,0)
\rput[t](2,-0.2){$\PS^2/\PS^1$}
\rput(1,0.5){$X_1$}\rput(3,0.5){$X_2$}
\rput(2,1){$Z$}\rput(2,2){$X_3$}
\rput(1.3,1.5){$V_1$}\rput(2.7,1.5){$V_2$}
\psdots[dotscale=1](7,0)(7,3)(9,3)(11,3)(11,0)
\psframe(7,0)(11,3)
\psline[linewidth=0.4pt](7,0)(9,3)(11,0)
\psline[linewidth=0.4pt](7,0)(9.5,1.75)(11,3)
\psline[linewidth=0.4pt](7,3)(8.5,1.75)(11,0)
\rput(7.5,1.75){$V_1$}
\rput(10.5,1.75){$V_2$}
\rput(8,2.5){$X_1$}\rput(10,2.5){$X_2$}
\rput(9,2){$Z$}\rput(9,0.5){$V_3$}
\rput(8.25,1.25){$X_3$}\rput(9.75,1.25){$X_4$}
\end{pspicture}
\end{center}\vspace{-10pt}
\caption{}
\vspace{-10pt}
\end{figure}
\end{enumerate}
\end{exa}
\begin{rem} \label{rem:4.3}Example~\ref{exa:4.2} shows that both types of elementary relations in the Sarkisov Program occur among the end products of the MMP of Fano $3$-folds with $\rho=3$. Note that the end products of the MMP on a Picard rank $3$ del Pezzo surface only give an example of a relation of type A (see Example~\ref{exa:4.1'} and \eqref{eq:1.2}). Since (1) and (25) in the classification of \cite{MM82} are toric, both types of relations actually occur among the end products of the MMP on \emph{toric Fano $3$-folds}. Last, $\PS^3/\{P\}$ is among the end products of the MMP on (11) and of (25), and this shows that both types of relations are realized in the Cremona group of $\PS^3$.   
\end{rem}
\subsection{General case} I now make precise the statement that relations in the Sarkisov Program are \emph{generated} by relations of the type occuring among the end products of the MMP on Picard rank $3$ Fano $3$-folds. 
\begin{thm}\label{thm:4.1}
The edge path group $E(\mcal N_{\mathfrak R})$ is generated by edge loops $\{\gamma_\mcal A \mid E(\resi_{\mcal A}\mcal N_\mathfrak R)= \Z[\gamma_{\mcal A}] \}$ as $\mcal A$ ranges over $(\rho(Z)-3)$-dimensional cones in \eqref{eq:3.2} that contain no big divisor.  
\end{thm}
\begin{cor}\label{cor:4.5}
Relations in the Sarkisov Program are generated by elementary relations. These are compositions of finitely many chains of type A or B, and there are examples of relations of type A and B induced by the MMP on nonsingular Fano $3$-folds of Picard rank $3$. 
\end{cor}
\begin{proof}
Let $I', I''\subset \mcal B^+(I)$ be such that
$I'= \{ i\in \mcal B^+(I) \mid \dim \mcal A_i= \rho(Z)-3\}$, and  
$I''= \{ i\in \mcal B^+(I) \mid \dim \mcal A_i= \rho(Z)-2, \mcal A_j\not \subset \overline{\mcal A_i}  \mbox{ for } j\in I'\}$. 
For each $i\in \mcal B^+(I)$ and cone $\mcal A_i$ in \eqref{eq:3.4}, let $\resi_{\mcal A_i} \mathfrak R$ be the $(\rho(Z)-d)$-dimensional residual ring as in Definition-Lemma~\ref{dfnlm:3.1}. Then, the $1$-skeleton of $\mcal N_\mathfrak R$ is:
 $$\mcal N_{\mathfrak R}^{(1)}= \bigcup_{i\in I'} \resi_{\mcal A_i} \mcal N_\mathfrak R \cup \bigcup_{i\in I''} \resi_{\mcal A_i}\mcal N_\mathfrak R.$$
By Corollary~\ref{cor:4.2}, $E(\resi_{\mcal A_i}\mcal N_\mathfrak R)$ is trivial for all $i\in I"$, and by Corollary~\ref{cor:4.3}, $E(\resi_{\mcal A_i}\mcal N_\mathfrak R)$ is either trivial or of the form $\Z[\gamma_{\mcal A_i}]$ for $i\in I'$. 
Since $\mcal N_\mathfrak R$ is connected, by the Van Kampen theorem, $E(\mcal N_\mathfrak R^{(1)})$ is generated by the groups $E(\resi_{\mcal A_i} \mcal N_\mathfrak R)$. 
Corollary~\ref{cor:4.5} then follows immediately from Corollary~\ref{cor:4.4}, Example~\ref{exa:4.2} and Remark~\ref{rem:4.3}. 
\end{proof}
\begin{rem}\label{rem:4.2}
Note that $E(\mcal N_\mathfrak R)$ is not in general \emph{freely generated} by the loops $\gamma_{\mcal A_i}$ when $\dim \mcal C_\mathfrak R>3$. The loops $\gamma_{\mcal A_i}$ generate the free group $E(\mcal N_\mathfrak R^{(1)})$. When $\dim \mcal C_\mathfrak R>3$, $\mcal N_{\mathfrak R}\neq \mcal N_\mathfrak R^{(1)}$, and the relations in $E(\mcal N_\mathfrak R)$ depend on the $2$-skeleton of $\mcal N_\mathfrak R$, that is by faces $\mcal A$ of dimension $\rho(Z)-4$.  
For example, there is a relation between $\gamma_{\mcal A_1}$ and $\gamma_{\mcal A_2}$ when $\overline{\mcal A_3}=\overline{\mcal A_1}\cap \overline{\mcal A_2}$ has dimension $\rho(Z)-4$, so that $\resi_{\mcal A_1}\mcal N_\mathfrak R$ and $\resi_{\mcal A_2}\mcal N_\mathfrak R$ both are subcomplexes of $\resi_{\mcal A_3}\mcal N_\mathfrak R$. 
In general, these relations between relations in the Sarkisov Program can be more complicated, but they are accounted for by divisorial rings with $4$-dimensional support. 
\end{rem}
\begin{exa}\label{exa:4.3}
To conclude, I determine the edge loops on $\mcal N_{\Effb S}$, for $S$ a smooth del Pezzo surface $S$ with $\rho(S)=4$. 
Let $S$ be the blow up of $\PS^2$ in $3$ points $\{P_1,P_2,P_3\}$; then $\mcal C_\mathfrak R=\Effb(S)=\Eff(S)$ is a rational polyhedral cone of dimension $4$.  Let $E_1, E_2, E_3$ be the $(-1)$-curves and $L_1, L_2,L_3$ the proper transforms of lines through two of the $3$ points $P_i$ on $\PS^2$. Let $S\to S_i$ (resp~$S\to T_i$) be the contraction of $E_i$ (resp~of $L_i$). 
The vertices of $\mcal N_{\mathfrak R}$ are:
\begin{itemize}
\item[] $(1)$ (resp.~$(14)$) corresponds to the ample model $S\lto \PS^2\to {pt}$ factoring through the three maps $S\lto S_i$ (resp.~$S\lto T_i$),
\item[] $(2,3,4)$ (resp.~$(11,12,13)$) correspond to the three ample models $S\lto \F_1\to \PS^1$ factoring through $S\lto S_i$ (resp. $S\lto T_i$),
\item[] $(5,6)$ (resp.~$(7,8)$, resp.~$(9,10)$) are the ample models $S\lto \PS^1\times \PS^1\to \PS^1$ factoring through $S\lto S_1$ and $S\lto T_1$ (resp.~$S_2, T_2$ and $S_3,T_3$).
\end{itemize}
Figure 6 illustrates the $1$-skeleton of $\mcal N_\mathfrak R$. 
\begin{figure}\begin{center}
\vspace{-10pt}
\xygraph{
!{<0cm,0cm>; <1cm,0cm>:<0cm,1cm>::}
!{(0,0) }*+{\bullet_{1}}="1"
!{(-1,1) }*+{\bullet_{2}}="2"
!{(1,.7) }*+{\bullet_{3}}="3"
!{(0.3,-1) }*+{\bullet_{4}}="4"
!{(-.8,2) }*+{\bullet_{5}}="5"
!{(1,1.8) }*+{\bullet_{6}}="6"
!{(-2.2,0.2) }*+{\bullet_{7}}="7"
!{(-1,-1) }*+{\bullet_{8}}="8"
!{(2,0.1) }*+{\bullet_{9}}="9"
!{(1.5,-1.2) }*+{\bullet_{10}}="10"
!{(3.5,-1.2) }*+{\bullet_{14}}="11"
!{(4.5,-2) }*+{\bullet_{11}}="12"
!{(3,0.1) }*+{\bullet_{12}}="13"
!{(2.8,-2) }*+{\bullet_{13}}="14"
"1"-"2" "1"-"3" "1"-"4"
"2"-"5" "2"-"7" "3"-"6" "3"-"9" "5"-"6" "4"-"8" "4"-"10" "7"-"8" "9"-"10" "9"-"13" "11"-"13" "11"-"14" "11"-"12" "10"-"14"
"5"-@/^1.8cm/"12" "6"-@/^/"13" "8"-@/_/"14" "12"-@/^1.45cm/"7" }
\end{center}
\vspace{-10pt}
\caption{}\vspace{-10pt}
\end{figure}
Here, there are $9$ faces of $\mcal B^+(\mcal C_\mathfrak R)$ of dimension $1$, $\mcal N_\mathfrak R= \cup_{i\in I'}\rest_{\mcal A} \mcal N_{\mathfrak R}$ and the associated cycles have the following interpretation:
\begin{itemize}
\item[] $(1,2,5,6,3)$ (resp. $(1,3,9,10,4)$, resp. $((1,2,7,8,4))$) corresponds to relations among the end products of the MMP on $S_1$ (resp.~on $S_2$, resp.~on $S_3$), 
\item[] $(12,9,10,13,14)$ (resp. $(11,5,6,12,14)$, resp. $((11,7,8,13,14)$) corresponds to relations among the end products of the MMP on $T_1$ (resp.~on $T_2$, resp.~on $T_3$), 
\item[] $(2,5,7,11), (3,9,12,6)$ and $(4,8,10,13)$ correspond to relations among the end products of the MMP of $S$ over each of the three $\PS^1$s that arise as ample models $S\lto \PS^1$ for some $D\in \Effb(S)$, 
\end{itemize}
As a by-product, this example shows how to recover all factorisations of the quadratic involution $\PS^2\dashto \PS^2$; these correspond to edge-loops through the $2$ vertices of $\mcal N_{\mathfrak R}$ associated to the $2$ distinct maps $S\lto \PS^2\to \{P\}$. 
\end{exa}

\bibliographystyle{amsalpha}
 \bibliography{biblio}
 \end{document}

%% file: macros
\theoremstyle{plain}

\newtheorem{thm}{Theorem}[section]

\newtheorem{pro}[thm]{Proposition}

\newtheorem{lem}[thm]{Lemma}

\newtheorem{cor}[thm]{Corollary}

\theoremstyle{definition}

\newtheorem{dfn}[thm]{Definition}

\newtheorem{nt}[thm]{Notation}
\newtheorem{conv}[thm]{Convention}

\newtheorem{dfnlm}[thm]{Definition-Lemma}
\newtheorem{rem}[thm]{Remark}

\newtheorem{exa}[thm]{Example}
\newtheorem{setup}[thm]{Set up}

\theoremstyle{remark}

\newcommand{\lto}{\longrightarrow}
\newcommand{\Z}{\mathbb{Z}}
\newcommand{\N}{\mathbb{N}}
\newcommand{\C}{\mathbb{C}}
\newcommand{\R}{\mathbb{R}}
\newcommand{\Q}{\mathbb{Q}}
\newcommand{\PS}{\mathbb{P}}
\newcommand{\F}{\mathbb{F}}

\newcommand{\vphi}{\varphi}

\newcommand{\dashto}{\dashrightarrow}
\newcommand{\mcal}{\mathcal}
\DeclareMathOperator{\Ner}{Ner}
\DeclareMathOperator{\Ob}{Ob}
\DeclareMathOperator{\Mor}{Mor}

\DeclareMathOperator{\id}{Id}

\DeclareMathOperator{\inte}{int}

\DeclareMathOperator{\Bs}{Bs}

\DeclareMathOperator{\Exc}{Exc}

\DeclareMathOperator{\Supp}{Supp}

\DeclareMathOperator{\Fix}{Fix}

\DeclareMathOperator{\Div}{Div}
\DeclareMathOperator{\rest}{Res}
\DeclareMathOperator{\resi}{resi}
\DeclareMathOperator{\Cone}{Cone}

\DeclareMathOperator{\Proj}{Proj}

\DeclareMathOperator{\Aut}{Aut}

\DeclareMathOperator{\Mob}{Mob}
\DeclareMathOperator{\Pic}{Pic} 
\DeclareMathOperator{\W}{Div}

\DeclareMathOperator{\Nef}{Nef}
\DeclareMathOperator{\Amp}{Amp}
\DeclareMathOperator{\Effb}{\overline{\mathrm{Eff}}}
\DeclareMathOperator{\Eff}{\mathrm{Eff}}

\DeclareMathOperator{\B}{Big}

\DeclareMathOperator{\ddiv}{div}
\DeclareMathOperator{\sB}{\mathbf{B}}